\documentclass[a4paper,12pt]{amsart}
\usepackage[utf8]{inputenc}
\usepackage[T1]{fontenc}
\usepackage[UKenglish]{babel}
\usepackage[margin=18mm]{geometry}

\usepackage{color}

\usepackage{graphicx}

\usepackage{amsmath,amssymb,amsfonts,amsthm,amscd,amsxtra}
\usepackage{mathrsfs,eucal,dsfont}
\usepackage{verbatim,enumitem}

\usepackage{hyperref,url}

\renewcommand{\geq}{\geqslant}

\newcommand{\R}{\mathds R}

\newcommand{\I}{\mathds 1}

\def\aa{\alpha}
\def\dd{\delta}
\def\d{{\rm d}}
\def\<{\langle}
\def\>{\rangle}

 \def\ss{\sqrt}
\def\bb{\beta}

\def\R{\mathbb R}   \def\ss{\sqrt} 
 \def\kk{\kappa} 
\def\dd{\delta}  \def\vv{\varepsilon} 
\def\<{\langle} \def\>{\rangle}  
    
\def\d{\text{\rm{d}}} \def\bb{\beta} \def\aa{\alpha} 
  \def\si{\sigma} 
 \def\beq{\begin{equation}}  
 
\def\e{\text{\rm{e}}}  \def\OO{\Omega}  
  
 \def\P{\mathbb P}

\def\E{\mathbb E}

\def\to{\rightarrow}
\def\8{\infty}\def\3{\triangle}
\def\1{\lesssim}

\def\iff{\iffalse}
\renewcommand{\bar}{\overline}
\renewcommand{\hat}{\widehat}
\renewcommand{\tilde}{\widetilde}

\newtheorem{theorem}{Theorem}[section]

\newtheorem{proposition}[theorem]{Proposition}

\theoremstyle{definition}

\newtheorem{remark}[theorem]{Remark}

\numberwithin{equation}{section}
\begin{document}
\allowdisplaybreaks

\title[SDEs with irregular drifts] {limit theorems for SDEs with irregular drifts}

\author{
Jianhai Bao\qquad
Jiaqing Hao}
\date{}
\thanks{\emph{J.\ Bao:} Center for Applied Mathematics, Tianjin University, 300072  Tianjin, P.R. China. \url{jianhaibao@tju.edu.cn}}

\thanks{\emph{J.\ Hao:}
Center for Applied Mathematics, Tianjin University, 300072  Tianjin, P.R. China. \url{hjq_0227@tju.edu.cn}}

\maketitle
\begin{abstract}
In this paper, concerning SDEs with H\"older continuous drifts, which are merely dissipative at infinity, and SDEs
with  piecewise continuous drifts, we investigate
 the strong law of large numbers and the central limit theorem for underlying additive functionals  and reveal the corresponding rates of convergence. To establish the  limit theorems under consideration, the exponentially contractive property of solution processes under the (quasi-)Wasserstein distance
plays  an indispensable role.  In order to  achieve such contractive property, which is new and interesting in its own right   for SDEs with H\"older continuous drifts or piecewise continuous drifts,
the reflection coupling method is employed and meanwhile a sophisticated test function  is built.

\medskip


\smallskip

\end{abstract}

\section{Introduction and main results}
The research on limit theorems for Markov processes has a long and rich history. As two typical candidates of  limit theorems,
 the law of large numbers (LLN for short) and the central limit theorem (CLT for abbreviation), depicting respectively the temporal average  convergence
  to the ergodic limit and the normalized fluctuations around the ergodic limit, have  developed greatly in various settings in a century; see \cite{HH,JS,KLO,Kulik}, to name just a few.

 For the sake  of  the establishment on  limit theorems concerned with Markov processes, one of the essential ingredients is to investigate  the  corresponding  ergodic property. When the Markov process under consideration  possesses the strong mixing properties  (e.g., ergodic in the total variation distance),
 the  LLN and the CLT associated with the additive functionals can be derived  with   the respective  convergence rates    $t^{-\frac{1}{2}+\vv}$ for any $\vv\in(0,\frac{1}{2})$ and $t^{-\frac{1}{2}}$; see, for instance, \cite[p.217-218]{Sh} and \cite[Theorem 5.1.2]{Kulik} for more details.  Whereas, in some occasions, the Markov process  under investigation  does not enjoy the strong mixing properties; see, for example, \cite[Example 5.1.3]{Kulik} concerned with functional SDEs, which  have  the so-called reconstruction property. Concerning  this setting (e.g., ergodicity  under  the (quasi)-Wasserstein distance   in lieu of the total variation distance), the study on limit theorems has also advanced
in the past few years. In particular,  the weak LLN and the CLT were explored in \cite{KW}  for weakly ergodic Markov processes by examining respectively the Feller property, exponential ergodicity under the $1$-Wasserstein distance and   uniform moment estimates with 
high order.   Additionally, \cite[Theorem 5.3.3]{Kulik} and \cite[Theorem 5.3.4]{Kulik} addressed respectively the issues on the LLN and the CLT for stationary Markov processes with  weakly ergodic properties. Subsequently,  under the continuous-time path coupling condition (see \cite[(5.3.10)]{Kulik} therein),  \cite[Proposition 5.3.5]{Kulik} extended the framework  in \cite[Theorems 5.3.3 and 5.3.4]{Kulik}
 to the non-stationary setup.
  In   comparison with the counterparts in
  \cite{KW,Kulik},  Shirikyan \cite{Sh} provided much more elegant conditions for the validity of the LLN and the CLT.  Particularly,
 Shirikyan \cite{Sh} formulated respectively a general criterion to establish the strong LLN and the CLT for weakly mixing Markov processes and, most importantly,   the associated convergence rates were provided therein. More precisely,
 \cite[Theorem 2.3]{Sh} shows that  the convergence rate of the strong LLN is $t^{-\frac{1}{2}+r_v}$, where $r_v:=q\vee(({1+v})/({4p}))$ for $v\in(0,2p-1)$ and $q<1/2$. So, from a quantitative point of view, the appearance of the quantity $q$ will attenuate the convergence rate in a certain sense. On the other hand, in
  \cite[Theorem 2.8]{Sh}, one of the sufficient conditions for   the CLT is concerned with the requirement on the uniform  moment estimates of exponential type (see \cite[(2.25)]{Sh} for related details) associated with weakly mixing Markov processes.  In most of the circumstances, such kind of exponential estimates (uniform in time) is a  formidable  task to be implemented; see, for example, \cite[Lemma 2.1]{BWYa} for the functional SDEs as one of representatives  with the weakly ergodic property.

  As we mentioned above,  \cite{KW,Kulik,Sh} have established the LLN and the CLT for weakly ergodic Markov processes under different scenarios. In addition, the framework formulated in \cite{KW,Kulik,Sh} has been applied to (functional) SDEs/SPDEs with {\it regular coefficients}. Especially,
  as a byproduct of \cite{BWYb},    \cite{BWY} investigated the strong LLN and the CLT for a range of functional SDEs. Later, \cite{BWY} was extended in \cite{WWZ} to treat the setting regarding  functional SDEs with infinite memory.
  
  In recent years, the theory on  strong/weak well-posedness and   distribution properties (e.g., gradient estimates and Harnack inequalities) of SDEs with irregular drifts has been studied systematically (see e.g. \cite{Wang16,XXZZ}). Yet, the study on limit theorems for SDEs with irregular drifts is still vacant so far.
  Inspired by the aforementioned literature \cite{KW,Kulik,Sh} as well as \cite{BWY, WWZ}, in the present work we make an attempt  to investigate the LLN and the CLT for several class of  SDEs with irregular drifts (e.g., the H\"older continuous drifts and the piecewise continuous drifts). Most importantly, besides the establishment of the LLN and the CLT, another main goal in this work is to improve the convergence rate of the LLN in  \cite[Theorem 2.3]{Sh}  and weaken the technical condition concerned with uniformly exponential estimates imposed in \cite[Theorem 2.8]{Sh}. The above can be viewed as some   motivations of our present work.

Another motivation arises from the significant advancements of numerical limit theorems for SDEs/SPDEs with regular coefficients.
  Recently, as for SDEs/SPDEs with (semi-)Lipschitz continuous  coefficients, there are plenty of literature  on the LLN and the CLT; see e.g.
\cite{LX,PP} for SDEs approximated via the forward Euler-Maruyama scheme, \cite{Jin} with regard to SDEs discretized by the backward Euler-Maruyama method, and \cite{CDHZ} concerning semilinear SPDEs approximated via the spectral Galerkin method in the spatial direction and the exponential integrator in the temporal direction.  To the best of our knowledge, the study on the LLN and the CLT for numerical schemes corresponding to SDEs with irregular drifts is still infrequent. So, in this work, we aim to lay the theoretical  foundation on
 the  LLN and the CLT for SDEs with H\"older continuous or piecewise continuous drifts (which are representative SDEs with irregular drifts) so that  we can pave undoubtedly the way to investigating the LLN and the CLT for the numerical SDEs with irregular drifts.

Inspired by the existing  literature mentioned above, in this work we intend to address the LLN and the CLT for SDEs with H\"older continuous drifts, where one part of drifts is dissipative in the long distance, and  satisfies the monotone and Lyapunov conditions, respectively. Additionally,       the LLN and CLT for SDEs with piecewise continuous drifts will also be explored in detail. The preceding contents   will be elaborated progressively in the following three subsections.

\subsection{LLN for SDEs with H\"older continuous drifts: partial dissipativity }\label{subsection1}
In this subsection, we work on the following SDE on $\R^d:$
\begin{equation}\label{E1}
\d X_t=\big(b_0(X_t)+b_1(X_t)\big)\,\d t+\si(X_t)\,\d W_t,
\end{equation}
where $b_0,b_1:\R^d\to\R^d$, $\si:\R^d\to\R^d\otimes\R^d$, $(W_t)_{t\ge0}$ is a $d$-dimensional Brownian motion on the complete  filtered probability space $(\OO,\mathscr F,(\mathscr F_t)_{t\ge0},\P)$.


Concerning the drifts  $ b_0$ and $b_1$, we shall assume that
\begin{enumerate}
\item[(${\bf H}_b$)] $b_1:\R^d\to\R^d$ is locally Lipschitz and there exist constants $\lambda_1,\lambda_2,\ell_0>0$ such that
\begin{equation}\label{E3}
2\<x-y,b_1(x)-b_1(y)\>\le \lambda_1|x-y|^2\I_{\{|x-y|\le \ell_0\}}-\lambda_2|x-y|^2\I_{\{|x-y|\ge \ell_0\}},\qquad x,y\in\R^d;
\end{equation}
$b_0\in   C^\alpha(\R^d)$ for some  $\alpha\in(0,1)$, i.e., there exists a constant $K_1>0$ such that
\begin{equation}\label{E0}
|b_0(x)-b_0(y)|\le K_1|x-y|^\alpha,\qquad x,y\in\R^d.
\end{equation}

\end{enumerate}

With regard to the diffusion term $\si$, we shall suppose that
\begin{enumerate}
\item[(${\bf H}_\si$)] $\si:\R^d\to\R^d\times\R^d$ is Lipschitz continuous, that is, there is a constant $K_2>0$ such that
\begin{equation}\label{E*}
\|\sigma(x)-\sigma(y)\|_{\rm HS}^2\le K_2|x-y|^2, \qquad x,y\in\R^d,
\end{equation}
 and moreover
there exists a constant $\kk\ge1$ such that
\begin{equation}\label{E4}
\frac{1}{\kk}|y|^2\le \<(\si\si^*)(x)y,y\>\le \kk|y|^2,  \qquad x,y\in\R^d.
\end{equation}
\end{enumerate}

Before   proceeding, we make some comments on Assumptions (${\bf H}_b$) and (${\bf H}_\si$), respectively.

\begin{remark}
 In literature, the Assumption \eqref{E3}
 is also named as  the dissipativity at infinity; see, for example, \cite{MMS}. 
  To demonstrate the condition \eqref{E3}, we provide an example below. 
Define  for some parameters $a>0$ and $n\ge1$,
\begin{equation*}
U(x)=x^2(g_n(x))^2+a^2-2axg_n(x),\qquad x\in\R,
\end{equation*}
where $g_n(x):=(x\wedge n)\vee (-n),\, x\in \R.$
Then, $b_1(x)=-U'(x)$ satisfies  the condition   \eqref{E3} (see, for example, \cite{MMS}) rather than the global convexity assumption: for some constant $K>0,$
\begin{equation*}
2\<x-y,b_1(x)-b_1(y)\>\le-K|x-y|^2,\qquad x,y\in\R.
\end{equation*}
 Since the drift term $b_0$ is singular (i.e. H\"older continuous), the uniformly elliptic condition in
\eqref{E4} is vitally important in addressing the well-posedness of \eqref{E1}. Most importantly, the condition \eqref{E4} also plays a crucial role in
exploring the exponentially contractive property of the SDE \eqref{E1} via the reflection coupling method; see the proof of Proposition \ref{pro1} for related details.
\end{remark}

Under Assumptions $({\bf H}_b)$ and $({\bf H}_\si)$, the SDE \eqref{E1} admits a unique strong solution $(X_t)_{t\ge0}$. Indeed, to address the strong well-posedness, we can adopt the routine as follows: first of all,
 we shall show that the SDE \eqref{E1} has a unique local solution via the Zvonkin transformation and subsequently claim that the local solution is indeed a global one. In some occasions, we shall write $(X_t^x)_{t\ge0}$ instead of $(X_t)_{t\ge0}$ to highlight the dependence on the initial value $X_0=x\in\R^d. $ In the following part, we shall denote  $C_{\rm Lip}(\R^d)$ by the collection of all Lipschitz continuous functions $f:\R^d\to\R$.
 Moreover, for a function $f:\R^d\to\R $  and $\nu\in\mathscr P(\R^d)$ (i.e., the set of probability measures on $\R^d$),
 we shall adopt  the shorthand notation $\nu(f)=\int_{\R^d}f(x)\,\nu(\d x)$ in case of the integral $\nu(|f|)<\8.$

The following LLN   reveals the convergence rate of the additive functional $A_t^{f,x}:=\frac{1}{t}\int_0^tf(X_s^x)\,\d s$ associated with a range of SDEs, which might be dissipative merely in the long distance and, in particular,  allows one part of the drift terms involved to be  H\"older continuous.

\begin{theorem}\label{thm1}
Assume $($${\bf H}_b$$)$ and $($${\bf H}_\si$$)$. Then, for any   $f\in C_{\rm Lip}(\R^d)$ and $\vv\in(0,1/2)$, there exist a random time $T_\vv\ge 1$ and a constant $C>0$ $($dependent on the Lipschitz constant $\|f\|_{\rm Lip}$ and the initial value  $x$$)$ such that for all $t\ge T_\vv,$
\begin{equation}\label{EE1}
\big|A_t^{f,x}-\mu(f)\big|\le Ct^{-\frac{1}{2}+\vv}, 
\end{equation}
where $\mu\in\mathscr P(\R^d)$ stands for the unique invariant probability measure of $(X_t^x)_{t\ge0}$ solving \eqref{E1}.
\end{theorem}
Before the end of this subsection, we make some remarks and
 comparisons  with the existing literature.
\begin{remark}
With regard to the random time $T_\vv$ mentioned in Theorem \ref{thm1}, it indeed has any finite $m$-th moment; see, for example, \cite[Corollary 2.4] {Sh} for related details.
 In \cite{Sh}, a general framework was provided to establish the LLN for mixing-type Markov processes.
When the weight function therein is constant (which corresponds the setup we work on in the present paper), the observable involved must be bounded. Hence, the general criterion in  \cite{Sh} cannot be applied (at least) directly to handle the setting we are interested in, where the observable herein is unbounded.
We have to refine the proof of \cite[Theorem 2.3]{Sh}. In addition, in \cite[Theorem 2.3]{Sh}, the corresponding convergent rate is $t^{-\frac{1}{2}+r_v}$ for $r_v:=q\vee(({1+v})/({4p}))$ with any $q<{1}/{2}$ and $v\in(0,2p-1)$. Whereas, in the present scenario, the associated convergence rate is  $t^{-\frac{1}{2}+r_v}$, in which $r_v:= ({1+v})/({2p})$ for $v\in(0,p/2-1)$. Consequently, in a certain sense (in particular, $q$ is close enough to ${1}/{2}$), we drop the redundant  parameter $ q<{1}/{2}$
and
improve accordingly  the convergence rate derived in \cite[Theorem 2.3]{Sh}.
\end{remark}

\subsection{CLT for SDEs with H\"older continuous drifts: monotone and Lyapunov conditions}\label{subsection2}
In this subsection, we move forward to derive the CLT associated with SDEs with H\"older continuous drifts.

In the proof of Theorem \ref{thm1}, the exponential contractivity under the $1$-Wasserstein distance is one of the important factors.  Nevertheless,  the function $\varphi_f$, defined in \eqref{E23} below, is merely  Lipschitz continuous under the underlying quasi-metric rather than globally Lipschitz continuous. Hence, the exponential contractivity under the $ 1$-Wasserstein distance is insufficient to establish the corresponding  CLT via the martingale approach. Conversely, the exponential contractivity  under the quasi-Wasserstein distance (see Proposition \ref{pro2} below) is adequate for our purpose. In this setup, we can further weaken Assumption (${\bf H}_b$).

Throughout this subsection, we are
still interested in  the SDE \eqref{E1}, where Assumption $({\bf H}_\si)$ is the same as that in Subsection \ref{subsection1} whereas Assumption  $({\bf H}_b)$ is substituted with the counterpart  $({\bf H}_b')$ below. More precisely,
\begin{enumerate}
\item[(${\bf H}_b'$)] $b_0\in C^\alpha(\R^d)$ satisfying \eqref{E0} and  there exist constants $\lambda,\lambda^*>0$ such that for all $x,y\in\R^d$,
\begin{equation}\label{E*7}
 2\<x-y,b_1(x)-b_1(y)\>\le \lambda|x-y|^2
\end{equation}
and
\begin{equation}\label{E*6}
\<x,b_1(x)\>\le -\lambda^* |x|^2+C_{\lambda^*}.
\end{equation}
\end{enumerate}

The condition \eqref{E*7} shows that the drift term $b_1$ satisfies the classical monotone condition, which, in literature,  is also called the one-sided Lipschitz condition.  Let $b_1(x)=-x+f(x), x\in\R^d,$ where the bounded function $f:\R^d\to\R^d$ is   Lipschitz with Lipschitz constant greater than $1.$ Obviously, both \eqref{E*7} and \eqref{E*6} are satisfied. Nevertheless, the condition \eqref{E3} doesn't hold any more.
\eqref{E*6}, in addition to \eqref{E0} and \eqref{E4},  indicates that the SDE \eqref{E1} fulfils the Lyapunov condition.
Under (${\bf H}_b'$) and $({\bf H}_\si)$,  the SDE \eqref{E1} has a unique strong solution $(X_t^x)_{t\ge0}$ with $X_0=x\in\R^d$
 and admits a unique invariant probability measure $\mu$ (see Proposition \ref{pro2}  below).

  Before we present the second main result concerning the CLT, we further need to introduce some additional notation.
For $p\ge2$ and $\theta\in(0,1]$,
denote $C_{p,\theta}(\R^d)$ by the family  of all continuous functions $f:\R^d\to\R$ such that
\begin{equation*}
\|f\|_{p,\theta}:=\sup_{x\neq y,x,y\in\R^d}\frac{|f(x)-f(y)|}{\psi_{p,\theta}(x,y)}<\8,
\end{equation*}
where for any $x,y\in\R^d,$
\begin{equation}\label{R1}
\psi_{p,\theta}(x,y):=(1\wedge|x-y|^\theta)(1+|x|^p+|y|^p).
\end{equation}
As for $f\in C_{p,\theta}(\R^d)$, we define the corrector $R_f$ as below
\begin{equation*}
R_f(x)=\int_0^\8\big((P_tf)(x)-\mu(f)\big)\,\d t,\quad x\in\R^d,
\end{equation*}
where $(P_tf)(x):=\E f(X_t^x)$ is the Markov semigroup associated with the solution process $(X_t^x)_{t\ge0}$. Moreover, we set
\begin{equation}\label{E23}
\varphi_f(x):=\E\Big|\int_0^1f(X_r^x)\,\d r+R_f(X_1^x)-R_f(x)\Big|^2,\qquad x\in\R^d.
\end{equation}
In terms of \cite[Lemma 4.1]{BWY},   $0\le\si^2_*:={\mu(\varphi_f)}<\8$ for $f\in C_{p,\theta}(\R^d)$.  The quantity $\si$   can be used to characterize  the asymptotic variance of the additive functional $\bar A_t^{f,x}:=\frac{1}{\ss t}\int_0^tf(X_s^x)\,\d s$ for $f\in C_{p,\theta}(\R^d)$.

Next, we  present another main  result, which is  concerned with the convergence rate of the CLT corresponding the additive functional $\bar A_t^{f,x}$. In detail, we have the following statement.

\begin{theorem}\label{CLT}
Assume $($${\bf H}_b'$$)$ and $($${\bf H}_\si$$)$.  Then, for any $f\in C_{p,\theta}(\R^d)$ with $\mu(f)=0$,
 $\si^2_*:={\mu(\varphi_f)}\ge0$ and $\vv\in(0,\frac{1}{4})$, there exists a constant $C_0=C_0(\|f\|_{p,\theta},\si_*,|x|)>0$ such that
\begin{equation}
\sup_{z\in\R^d}\big(\theta_{\si_*}(z)\big|\P(\bar A_t^{f,x}\le z)-\Phi_{\si_*}(z)\big|\big)\le C_0t^{-\frac{1}{4}+\vv},\quad t\ge 1,
\end{equation}
where $\theta_{\si_*}(z):= \I_{\{0<\si_*<\8\}} +(1\wedge |z|)\I_{\{\si_*=0\}} $ and $\Phi_{\si_*}(z)$ stands for the centered Gaussian distribution function with variation $\si^2_*.$
\end{theorem}

Before the end of this subsection, we  make some further remarks.
\begin{remark}
As Theorem \ref{thm1}, Theorem \ref{thm2} is applicable to an SDE  with the H\"older continuous drift, where another part of the drift is monotone and satisfies the Lyapunov condition.
In \cite[Theorem 2.8]{Sh}, a general criterion  was provided to explore the CLT for uniformly mixing Markov families.
In particular, the uniform  moment estimates of exponential type (see \cite[(2.25)]{Sh}), as one of the sufficient conditions, was imposed therein.
However, such a uniform moment estimate   is, in general,   hard to check; see \cite{BWY} for the setting on functional SDEs and numerical SDEs.
In lieu of the requirement on the uniform exponent moment, the uniform moment in the polynomial type, which is much easier to verify,   is sufficient for our purpose as shown in the proof of Theorem \ref{thm2}.
\end{remark}

\subsection{LLN and CLT for SDEs with piecewise  continuous drifts}\label{subsection3}
In the previous two subsections, as far as two different setups are concerned, we establish respectively the strong LLN and the CLT. Whereas, no matter what which setting, the continuity of the drift terms is necessary.

 For   the objective in this subsection,  we consider an illustrative example:
 \begin{equation*}
 b(x):=-|x|-2,\quad |x|\ge 1;\quad  b(x):=1-x^2,\quad |x|<1
 \end{equation*}
and $\si(x):=\frac{1}{2}(1+\frac{1}{1+x^2})$. Apparently, the drift term $b$ above no longer  satisfies Assumption (${\bf H}_b$) or (${\bf H}_b'$)   due to the appearance of the discontinuous points. Yet, the drift $b$ is globally dissipative in the long distance, the associated SDE should be ergodic under an appropriate probability (quasi-)distance, which is still vacant  to the best of our knowledge.   Therefore, intuitively speaking, the corresponding strong LLN and the CLT should be valid. So, in this subsection, our goal is  to deal with  the strong LLN and the CLT for SDEs, where the drifts involved might be discontinuous. So far, the topic mentioned above is still rare.

To explain the underlying essence to handle the limit theorems for SDEs with discontinuous drifts and, most importantly,  avoid the cumbersome notation,  we shall consider the scalar SDE
\begin{equation}\label{W*1}
\d X_t=b(X_t)\,\d t+\si(X_t)\,\d W_t,
\end{equation}
where $b:\R\to\R$ is piecewise continuous, i.e., there exist finitely many points $\xi_1<\cdots<\xi_k$  such that $b$ is continuous respectively on the intervals   $I_{i}:=(\xi_{i},\xi_{i+1})$, $i\in \mathbb S_k:=\{0,1,\cdots,k\}$,  where $\xi_0=-\8$ and $\xi_{k+1}=\8$; $\si:\R\to\R$
is continuous;
$(W_t)_{t\ge0}$ is a $1$-dimensional Brownian motion. So far, the SDE \eqref{W*1} with discontinuous drifts has been applied extensively in e.g. stochastic control theory and mathematical finance.

Besides Assumption (${\bf H}_\si$) with $d=1$,
we shall assume that
\begin{enumerate}
\item[$({\bf A}_b)$] For each integer $n\ge |\xi_1|\vee|\xi_k|$, there exists an increasing function $\phi:[0,\8)\to [0,\8)$
\begin{equation}\label{W*5}
|b(x)-b(y)|\le \phi(n)|x-y|,\quad x,y\in I_i\cap B_n(0),\quad i\in\mathbb S_k,
\end{equation}
where  $B_n(0):=\{x\in\R:|x|\le n\}$,
and there exists a constant $\lambda_{\star }>0$ such that
\begin{equation*}
(x-y)(b(x)-b(y))\le \lambda_{\star }(x-y)^2,\quad x,y\in I_i,\quad i\in\mathbb S_k.
\end{equation*}
Moreover,
there are constants $\lambda^\star, C_{\lambda^\star}>0$ and $\vv^\star\in(0,1/2]$ such that
\begin{equation}\label{W*}
\vv^\star|b(x)|(1+|x|)+xb(x)\le C_\star-\lambda^\star x^2,\qquad x\in\R.
\end{equation}
\end{enumerate}

Below, with regard to Assumptions $({\bf A}_b)$ and (${\bf H}_\si$) with $d=1$, we make the following remarks.

\begin{remark}
In contrast to \eqref{E*6}, the Lyapunov condition \eqref{W*}   is  a little bit unusual. This condition is imposed naturally when we handle the ergodicity of the transformed SDE; see the proof of Proposition \ref{pro3} for more details. Most importantly, the appearance of $\vv^\star\in(0,1/2]$
allows $b$ to be highly nonlinear.

In \cite[Lemma 3]{MSY}, the strong well-posedness of \eqref{W*1} was treated under the global Lyapunov condition (see (A1) therein), the local monotonicity (see (A2)(i) therein), the locally polynomial growth condition (see (A2)(ii) therein) as well as  the globally polynomial growth condition (see (A3) therein) concerned with the diffusion term, which is also  non-degenerate at the discontinuous points of the drift term (see (A4) therein). By a close inspection of the argument of \cite[Lemma 3]{MSY}, the much weaker condition \eqref{W*5} can  indeed take the place of
 (A2)(ii).  Therefore, by following the exact lines of \cite[Lemma 3]{MSY}, the SDE \eqref{W*1} is strongly well-posed under Assumptions $({\bf A}_b)$ and (${\bf H}_\si$) with $d=1.$
  If we are only concerned with the well-posedness of \eqref{W*1}, Assumption (${\bf H}_\si$) with $d=1$ can be replaced definitely by
(A3) and (A4). Whereas,  the little bit strong condition (${\bf H}_\si$) with $d=1$, compared with  (A3) and (A4),
is imposed to achieve the exponentially contractive property under the quasi-Wasserstein distance via  the reflection coupling approach.
\end{remark}

Our third main result concerned with the strong LLN and the CLT for the SDE \eqref{W*1} with piecewise continuous drifts is stated as follows.
\begin{theorem}\label{thm2}
Assume $({\bf A}_b)$ and $({\bf H}_\si)$ with $d=1$. Then,
\begin{enumerate}
\item[$(1)$] $({\bf Strong~ LLN})$
For any $f\in C_{p,\theta}(\R)$ and $\vv\in(0,1/2)$, there exist a random time $T_\vv\ge 1$ and a constant $C=C(\|f\|_{p,\theta},|x|)>0$  such that for all $t\ge T_\vv,$
\begin{equation*}
\big|A_t^{f,x}-\mu(f)\big|\le Ct^{-\frac{1}{2}+\vv}, 
\end{equation*}
where $\mu \in\mathscr P(\R^d)$ is the unique invariant probability measure of $(X_t)_{t\ge0}$ solving \eqref{W*1}; see Proposition \ref{pro3} below.
 \item[$(2)$]$({\bf CLT})$ For any $f\in C_{p,\theta}(\R)$ with $\mu(f)=0$, and $\vv\in(0,1/4)$,
If
 $\si^2_*:={\mu(\varphi_f)}\ge0$,   there exists a constant $C_0=C_0(\|f\|_{p,\theta},\si_*,|x|)>0$ such that
\begin{equation*}
\sup_{z\in\R^d}\big(\theta_{\si_*}(z)\big|\P(\bar A_t^{f,x}\le z)-\Phi_{\si_*}(z)\big|\big)\le C_0t^{-\frac{1}{4}+\vv},\quad t\ge 1,
\end{equation*}
where $\theta_{\si_*}(z):= \I_{\{0<\si_*<\8\}} +(1\wedge |z|)\I_{\{\si_*=0\}} $.
\end{enumerate}
\end{theorem}

Before the ending of this subsection, we make some further comments.
\begin{remark}
To finish the proof of Theorem \ref{thm2}, the $1$-dimensional  diffeomorphism transformation (see \eqref{HH} below) plays a crucial role. For the multidimensional transformation to handle well-posedness and numerical approximations for SDEs with piecewise continuous drifts, we refer to \cite[Theorem 3.14]{LSb} for more details. With the help of the multidimensional transformation initiated in \cite{LSb},   Theorem \ref{thm2} can be generalized to the multidimensional SDEs with piecewise continuous drifts. Since such a generalization will only render the notation more cumbersome without bringing any new insights into the arguments, in the present work  we restrict ourselves to the $1$-dimensional setup.

Even though the original SDE under consideration is dissipative in the long distance, the corresponding  transformed SDE is no longer dissipative (at infinity). Based on this point of view, the SDE \eqref{W*1} is ergodic under the quasi-Wasserstein distance rather than the genuine Wasserstein distance; see Proposition \ref{pro3} for more details.
\end{remark}

The remainder of this paper is organized as follows. In Section \ref{sec2}, via the reflection coupling,  we investigate the $1$-Wasserstein exponential contractivity for SDEs with H\"older continuous drifts, which also allow the drifts involved to be dissipative in the long distance. Subsequently, the proof of Theorem \ref{thm1} is complete.  Section \ref{sec3} is devoted to the proof of Theorem \ref{thm1} based on the establishment of the exponential ergodicity under the quasi-Wasserstein distance, which is interesting in its own right for SDEs with H\"older continuous drifts, where the other drift parts satisfy  the monotone and Lyapunov condition.  In Section \ref{sec4}, we aim to complete the proof of Theorem \ref{thm2} with the aid of the exponential contractivity under the quasi-Wasserstein distance, which is new for SDEs with piecewise continuous drifts.

\section{Proof of Theorem \ref{thm1}}\label{sec2}
In this section, we aim to complete the proof of Theorem \ref{thm1}, which is based on the following $1$-Wasserstein contractive property. 
\begin{proposition}\label{pro1}
Assume $($${\bf H}_b$$)$ and $($${\bf H}_\si$$)$. Then, there exist constants $C^*,\lambda^*>0$ such that for all $t\ge0$ and $\mu,\nu\in\mathscr P_1(\R^d)$ $($the set of probability measures on $\R^d$ with finite  moments of first order$)$,
\begin{equation}\label{E14}
\mathbb W_1(\mu P_t, \nu P_t)\le C^*\e^{-\lambda^* t}\mathbb W_1(\mu,\nu),
\end{equation}
where $\mu P_t$ stands for the law of $X_t$, the solution to \eqref{E1},  with the initial distribution $\mathscr L_{X_0}=\mu$, and $\mathbb W_1$ means the $1$-Wasserstein distance. Moreover, \eqref{E14} implies that $(X_t)_{t\ge0}$ has a unique invariant probability measure $\mu.$
\end{proposition}

\begin{proof}
The proof,  based on the reflection coupling approach, of Proposition \ref{pro1} is inspired by the counterpart of \cite[Theorem 3.1]{Wang23}, which indeed was traced back to \cite{PW}. In \cite[Theorem 3.1]{Wang23}, an abstract framework upon exponential ergodicity of McKean-Vlasov SDEs, which are dissipative in the long distance, was presented. In the present setup, we follow essentially  the line in \cite[Theorem 3.1]{Wang23} whereas
we refine the corresponding details and provide   explicit conditions  imposed on the coefficients so the content is much more readable.

Due to \eqref{E4}, for each $x\in\R^d$, the matrix $(\si\si^*)(x)-\frac{1}{2\kk}I_{d\times d}$  is a nonnegative-definite symmetric matrix  so
 there exists a symmetric $d\times d$-matrix  $\tilde{  \si}(x)$ such that $\tilde{\si}(x)^2=(\si\si^*)(x)-\frac{1}{2\kk}I_{d\times d}$, where $I_{d\times d}$
means the $d\times d$-identity matrix. Therefore, we readily   have
$
(\si\si^*)(x)=\tilde{\si}(x)^2+\frac{1}{2\kk}I_{d\times d},\, x\in\R^d.
$
Consider the SDE
\begin{equation}\label{E8}
\d Y_t=b(Y_t)\d t+\tilde{\si}(Y_t)\d \tilde{W}_t+\frac{1}{\ss {2\kk}}\d \hat{W}_t,
\end{equation}
where  $(\tilde{W}_t)_{t\ge0}$ and $(\hat{W}_t)_{t\ge0}$ are mutually independent $d$-dimensional Brownian motions  defined on the same probability space.  To demonstrate that \eqref{E8} has a unique strong solution under Assumptions
$({\bf H}_b)$ and $({\bf H}_\si)$, we introduce  the notations
$$\hat \si(x):=\Big(\tilde{\si}(x),\frac{1}{\ss {2\kk}}I_{d\times d}\Big)\in \R^d\otimes\R^{2d},\quad x\in\R^d,   \quad    \mbox{ and } \quad \bar W_t:=(\tilde{W}_t, \hat{W}_t),$$
where $(\bar W_t)_{t\ge0}$ is a $2d$-dimensional Brownian motion. Whereafter,
\eqref{E8} can be reformulated as
\begin{equation*}
\d Y_t=b(Y_t)\d t+\hat{\si}(Y_t)\d \bar{W}_t.
\end{equation*}
As a result, in terms of \textcolor{blue}{Subsection \ref{subsection1}}, it is sufficient to examine that $\hat \si$ satisfies Assumption $({\bf H}_{\hat \si})$ so that  \eqref{E8} is strongly well-posed. Below, we aim to check the associated details, one by one. In the first place,
by invoking \eqref{E*} and \eqref{E4} , we derive that
\begin{equation}\label{E11}
\begin{split}
\|\hat \si(x)-\hat \si(y)\|_{\rm HS}
\le  2(K_2\kk^3)^{\frac{1}{2}} | x-y|,\qquad x,y\in\R^d,
\end{split}
\end{equation}
where
 we also used the fact that $\|\ss A-\ss B\|_{\rm HS}\le \frac{1}{2\lambda}\|A-B\|_{\rm HS}$
for    $d\times d$ symmetric positive matrices $A$ and $B$ with all eigenvalues greater than $\lambda>0$ (see, for example, \cite[(3.3)]{PW} for related details). 
 Therefore, \eqref{E11} enables us to conclude that the mapping $x\mapsto \hat \si(x)$ is also Lipschitz.  In the next place,  it is easy to see from \eqref{E4} that for all $x,y\in\R^d$,
\begin{equation*}
\frac{1}{\kk}|y|^2\le \<(\hat \si\hat \si^*)(x)y,y\>=\<(  \si  \si^*)(x)y,y\>\le   \kk |y|^2.
\end{equation*}

For ${\bf 0}\neq x\in\R^d$, set the normalized vector ${\bf n}(x):= x/|x|$ and define the orthogonal matrix
\begin{equation*}
\Pi_x=I_{d\times d}-2{\bf n}(x)\otimes {\bf n}(x)\in\R^d\times\R^d.
\end{equation*}
 To achieve the quantitative estimate \eqref{E14}, we work on the SDE
\begin{equation}\label{E10}
\begin{cases}
\d \hat Y_t=b(\hat Y_t)\d t+\tilde{\si}(\hat Y_t)\d \tilde{W}_t+\frac{1}{\ss {2\kk}}\Pi_{Z_t}  \d \hat{W}_t,\qquad t<\tau,\\
\d Y_t=b( Y_t)\d t+\tilde{\si}(  Y_t)\d \tilde{W}_t+\frac{1}{\ss {2\kk}}\d \hat{W}_t,\qquad\quad\quad~~ t\ge\tau,
\end{cases}
\end{equation}
where  the coupling time $
\tau:=\inf\big\{t\ge 0: Z_t={\bf 0}\big\}
$ with $Z_t:=Y_t-\hat Y_t$.
 Since  $\Pi_\cdot$ is an orthogonal matrix, \eqref{E10} is strongly well-posed before the coupling time as shown in the analysis above. Additionally,  \eqref{E10} coincides with \eqref{E8} after the coupling time. Thereby, \eqref{E10}
is strongly well-posed.

Owing to the existence of an optimal coupling, in the following context,  we can choose the initial values $Y_0$ and $\hat Y_0$ such that $\mathbb W_1(\mu,\nu)=\E|Y_0-\hat Y_0|$ for given $\mu,\nu\in\mathscr P_1(\R^d)$.
Applying It\^o's formula, we right now obtain from \eqref{E8} and \eqref{E10} that
\begin{equation}\label{E15}
\begin{split}
\d |Z_t|&\le \frac{1}{2|Z_t|}\big(2\<Z_t,b(Y_t)-b(\hat Y_t)\>+ \| \tilde{\si}(Y_t)-\tilde{\si}(\hat Y_t)\|_{\rm HS}^2\big)\,\d t\\
&\quad +\big\<{\bf n}(Z_t ), (\tilde{\si}(Y_t)-\tilde{\si}(\hat Y_t) )\d \tilde{W}_t\big\>+\frac{1}{\ss {\kk/2}}\big\<{\bf n}(Z_t ),\d\hat{W}_t\big\>,\quad t<\tau,
\end{split}
\end{equation}
where we also utilized the fact  that for any ${\bf 0}\neq x\in\R^d,$
\begin{equation*}
\<I_{d\times d}-{\bf n}(x)\otimes {\bf n}(x),{\bf n}(x)\otimes {\bf n}(x)\>_{\rm HS}=0.
\end{equation*}
By invoking \eqref{E11}, it follows from \eqref{E4} that
\begin{equation}\label{E11E}
\begin{split}
\|\tilde{  \si}(x)-\tilde{  \si}(y)\|_{\rm HS}
\le 4\kk^{\frac{13}{8}}d^{\frac{1}{8}}K_2^{\frac{3}{8}}|x-y|^{\frac{3}{4}},\quad x,y\in\R^d.
\end{split}
\end{equation}
With the aid of \eqref{E0} and \eqref{E*}, we find from \eqref{E11E} that
\begin{equation}\label{EE9}
\begin{split}
2&\<x-y,b(x)-b(y)\>+ \| \tilde{\si}(x)-\tilde{\si}(y)\|_{\rm HS}^2
\le \phi(|x-y|)|x-y|,\quad x,y\in\R^d,
 \end{split}
\end{equation}
where for $u\ge0$,
\begin{equation*}
\phi(u):=\big((\lambda_1+\lambda_2)u+2K_1u^\alpha+16\kk^{\frac{13}{4}}d^{\frac{1}{4}}K_2^{\frac{3}{4}}u^{\frac{1}{2}}\big)\I_{\{u\le \hbar_0\}}-\frac{1}{2}\lambda_2u
\end{equation*}
with
\begin{equation}\label{E*3}
\hbar_0:=\ell_0\vee\bigg(\frac{8K_1}{\lambda_2 }\bigg)^{\frac{1}{1-\alpha}}\vee\bigg(\frac{64\kk^{\frac{13}{4}}d^{\frac{1}{4}}K_2^{\frac{3}{4}}}{\lambda_2}\bigg)^2.
\end{equation}
Consequently,   \eqref{E15} yields
\begin{equation*}
\d |Z_t|\le \frac{1}{2 }\phi(|Z_t|)\,\d t +\<{\bf n}(Z_t ), (\tilde{\si}(Y_t)-\tilde{\si}(\hat Y_t) )\d \tilde{W}_t\>+\frac{1}{\ss {\kk/2}}\<{\bf n}(Z_t ),\d\hat{W}_t\>,\qquad t<\tau.
\end{equation*}

Define the test function
\begin{equation*}
f(r)= \kk \int_0^r\e^{- \frac{\kk}{2} \int_0^u\phi(v)\,\d v}\int_u^\8s\e^{ \frac{\kk}{2} \int_0^s\phi(v)\,\d v}\,\d s \d u,\quad r\ge0.
\end{equation*}
Straightforward calculations show that
\begin{equation*}
\begin{split}
f'(r)&= \kk  \e^{- \frac{\kk }{2}\int_0^r\phi(v)\,\d v}\int_r^\8s\e^{ \frac{\kk}{2} \int_0^s\phi(v)\,\d v}\,\d s,
  \quad r\ge0; \\
  f''(r)&=  \frac{\kk}{2} \big( - \phi(r)  f'(r)- 2 r\big), \quad r\in[0,\hbar_0),
\end{split}
\end{equation*}
and  that
\begin{equation*}
f'(r)=\frac{4}{ \lambda_2},\quad r\ge\hbar_0; \quad f''(r)=0, \quad r>\hbar_0.
\end{equation*}
Whence, there exist constants $c_*,c^{**}>0$ such that
\begin{equation}\label{E12}
c_*r\le f(r)\le c^{**}r,\qquad r\ge0
\end{equation}
and
\begin{equation}\label{E13}
\frac{1}{2}f'(r)\phi(r)+\frac{1}{ \kk}f''(r)=-r,\quad r\in[0,\hbar_0)\cup(\hbar_0,\8).
\end{equation}
Note that $f$ introduced above is a piecewise $C^2$-function. Thus,
Tanaka's formula, together with the continuity of $f'$,
shows that
\begin{equation*}
 \e^{\frac{t}{c^{**}}}f(|Z_t|) \le f(|Z_0|)+ \int_0^t\e^{\frac{s}{c^{**}}}\Big(\frac{1}{c^{**}}f(|Z_s|)+\frac{1}{2}f'(|Z_s|)\phi(|Z_s|)+\frac{1}{ \kk}f''(|Z_s|)\Big)\,\d s+  M_t,\quad t<\tau
\end{equation*}
for some martingale $(M_t)_{t\ge0}.$ This, combining \eqref{E12} with \eqref{E13},  further gives that
\begin{equation*}
\begin{split}
 \E\big(\e^{\frac{t\wedge \tau}{c^{**}}}f(|Z_{t\wedge\tau}|)\big) &\le \E f(|Z_0|)+\E \bigg( \int_0^{t\wedge\tau}\e^{\frac{s}{c^{**}}}\Big(\frac{1}{c^{**}}f(|Z_s|)+\frac{1}{2}f'(|Z_s|)\phi(|Z_s|)+\frac{1}{ \kk}f''(|Z_s|)\Big)\,\d s\bigg) \\
 &\le\E f(|Z_0|).
 \end{split}
\end{equation*}
Thanks to $f(|Z_t|)\equiv0$ for all $t\ge\tau$,  it is apparent  that for all $t\ge0,$
\begin{equation*}
\e^{\frac{t}{c^{**}}}\E f(|Z_t|)=\E\big(\e^{\frac{t\wedge \tau}{c^{**}}}f(|Z_{t\wedge\tau}|)\big)\le \E f(|Z_0|).
\end{equation*}
Finally, \eqref{E14} follows immediately by recalling $\mathbb W_1(\mu,\nu)=\E|Y_0-\hat Y_0|$ and taking \eqref{E12} into consideration.

Under \eqref{E3}, \eqref{E0} and \eqref{E4}, there exists a constant $c_\star>0$ such that $\sup_{t\ge 0}\E|X_t|^2\le c_\star(1+\E|X_0|^2)$; see \eqref{EE3} below for further details. Whence, the Krylov-Bogoliubov theorem yields that $(X_t)_{t\ge0}$ has an invariant probability measure. Thus, the contractive property \eqref{E14} implies the uniqueness of invariant probability measures.
\end{proof} 

Below, we move to finish the
\begin{proof}[Proof of Theorem \ref{thm1}]
In the sequel, we shall assume that  $f\in C_{\rm Lip}(\R^d)$    and set $\vv\in(0,1/2)$.
Evidently, it suffices to verify that \eqref{EE1} is valid as long as $\mu(f)=0$. Accordingly, we shall stipulate $\mu(f)=0$ in the subsequent analysis. Below, we shall write $X_t$ instead of $X_t^x$, set
$A_t(f):=\frac{1}{t}\int_0^tf(X_s)\,\d s$ for all $t\ge0,$ and use the notation $a\lesssim b$ for given $a,b\ge0$ provided that there exists a constant $c_0>0$ such that $a\le c_0b.$
Note that for any $t\ge0,$
\begin{equation*}
 |A_t(f) |
 \le  \big|A_{\lfloor t\rfloor}(f) \big| +\frac{1}{t}\int_{\lfloor t\rfloor}^t|f(X_s)|\,\d s,
\end{equation*}
where $ \lfloor t\rfloor$ denotes the integer part of $t\ge0.$
Whence, to obtain the assertion \eqref{EE1},
 it remains to show respectively that
there exists   a random time $  T\ge 1$ (dependent on $\vv$) such that
\begin{equation}\label{EEEE2}
\big|A_{\lfloor t\rfloor}(f) \big| \lesssim t^{-\frac{1}{2}+\vv}\qquad \mbox{ and } \qquad \frac{1}{t}\int_{\lfloor t\rfloor}^t|f(X_s)|\,\d s\lesssim t^{-\frac{1}{2} }\qquad \mbox{ for all } t\ge T.
\end{equation}

By \eqref{E3}, \eqref{E0} and \eqref{E4}, it follows that there exist constants $c_1,c_2>0$ such that for any $p\ge2,$
\begin{equation*}
\d |X_t|^p\le \big(-c_1|X_t|^p+c_2\big)\,\d t+p|X_t|^{p-2}\<X_t,\si(X_t)\d W_t\>.
\end{equation*}
Then, Gronwall's inequality implies that for all $t\ge 0$ and $x\in\R^d,$
\begin{equation}\label{EE3}
\E|X_t|^p\le \frac{c_2}{c_1}+|x|^p.
\end{equation}
For any integer $q\ge2,$ direct calculations show that
\begin{equation*}
\E\Big|\int_0^tf(X_s)\,\d s\Big|^q\lesssim\bigg(\int_0^t\int_0^s\big(\E\big( |f(X_u)|(|P_{s-u}f|)(X_u)\big)^{\frac{q}{2}}\big)^{\frac{2}{q}}\,\d u\d s\bigg)^{\frac{q}{2}},
\end{equation*}
see, for instance, \cite[(2.22)]{Sh}. Next, using the invariance of $\mu$ followed by
exploiting the Kontorovich dual
and applying Proposition \ref{pro1} yields that for all $t\ge 0$ and $x\in\R^d,$
\begin{equation}\label{ER}
\big|(P_tf)(x)-\mu(f)\big|\le \|f\|_{\rm Lip}\mathbb W_1(\delta_x P_t,\mu P_t)\le C^*\e^{-\lambda^* t}(|x|+\mu(|\cdot|)),
\end{equation}
where $\|f\|_{\rm Lip}$ is the Lipschitz constant of the function $f$. The previous estimate, in addition to the Lipschitz property of $f$, $\mu(|\cdot|)<\8$ as well as \eqref{EE3}, gives that
\begin{equation}\label{WW}
\E\Big|\int_0^tf(X_s)\,\d s\Big|^q\lesssim\bigg(\int_0^t\int_0^s\e^{-\lambda^*(s-u)}\big(1+\E|X_u|^q\big)^{\frac{2}{q}}\,\d u\d s\bigg)^{\frac{q}{2}}\lesssim t^{\frac{q}{2}}.
\end{equation}
Subsequently, by means of H\"older's inequality, we have for any $q\ge2,$
\begin{equation}\label{EE4}
\E |A_t(f) |^q \lesssim t^{-\frac{q}{2}}.
\end{equation}

For any integer $n\ge 1,$ by the Chebyshev inequality, together with \eqref{EE4} for $q=\frac{2}{\vv}>4$, we deduce that
\begin{equation*}
\P\big(|A_n(f)|>n^{-\frac{1}{2}+\vv}\big)\le n^{\frac{1}{\vv}-2} \E|A_n(f)|^{\frac{2}{\vv}}\lesssim n^{-2}.
\end{equation*}
As a consequence, the Borel-Cantelli lemma yields that  there exists a random variable $T_1 \ge 1$ such that
\begin{equation*}
\big|A_{\lfloor t\rfloor}(f)\big|\le {\lfloor t\rfloor}^{ -\frac{1}{2}+\vv},\quad \mbox{ a.s.},\qquad t\ge T_1.
\end{equation*}
This apparently ensures the first statement.

Next, we proceed to examine the second statement in \eqref{EEEE2}. In view of the BDG inequality,  we infer from \eqref{E4} and \eqref{EE3} that there exist constants $c_3,c_4>0$ such that for all integer $k\ge0$ and any $p\ge 2$,
\begin{equation*}
\begin{split}
\E\big(\sup_{k\le s\le  k+1}|X_s|^p\big)&\le c_3+\E|X_k|^p+p\,\E\Big(\sup_{k\le s\le k+1}\int_k^s|X_u|^{p-2}\<X_u,\si(X_u)\d W_u\>\Big)\\
&\le c_4(1+|x|^p)+\frac{1}{2}\E\big(\sup_{k\le s\le  k+1}|X_s|^p\big)
\end{split}
\end{equation*}
so that  for all integer $k\ge0$ and any $p\ge 2$,
\begin{equation}\label{EE5}
\E\big(\sup_{k\le s\le  k+1}|X_s|^p\big)\lesssim 1+|x|^p.
\end{equation}
By retrospecting that   $f:\R^d\to\R$ is of linear growth, there exists a constant $c^\star>0$ such that $|f(x)|\le c^\star(1+|x|)$ for all $x\in\R^d.$ It is ready to see that for any integer $k\ge 16,$
\begin{equation}\label{EEE1}
\begin{split}
\P\Big(\sup_{k\le t\le k+1}|f(X_t)|>c^\star k^{\frac{1}{4}}\Big) \le \P\Big(\sup_{k\le t\le k+1}| X_t |>k^{\frac{1}{4}}-1\Big) \le  \P\Big(\sup_{k\le t\le k+1}| X_t |>\frac{1}{2}k^{\frac{1}{4}} \Big).
\end{split}
\end{equation}
Then, the Chebyshev inequality, besides \eqref{EE5}, signifies that
\begin{equation}\label{EEE2}
\begin{split}
\P\Big(\sup_{k\le t\le k+1}|f(X_t)|>c^\star k^{\frac{1}{4}}\Big) \le  \frac{32}{k^{\frac{5}{4}}}\E\big(\sup_{k\le t\le k+1}| X_t |^5\big)\lesssim\frac{1}{k^{\frac{5}{4}}} (1+|x|^5).
\end{split}
\end{equation}
Once more, applying the Borel-Cantelli lemma enables us to derive that there exists a random variable $T_2\ge 16$ such that
$|f(X_t)|\le c^\star t^{\frac{1}{4}}$, a.s., for all $t\ge T_2.$ Therefore, we obtain that for all $t\ge T_2,$
\begin{equation*}
\frac{1}{t}\int_{\lfloor t\rfloor}^t|f(X_s)|\,\d s\lesssim \frac{1}{\lfloor t\rfloor}\int_{\lfloor t\rfloor}^ts^{\frac{1}{4}}\,\d s\lesssim \lfloor t\rfloor^{\frac{1}{4}}\Big(\big(1+\frac{1}{\lfloor t\rfloor}\big)^{ \frac{5}{4}}-1\Big),\quad \mbox{ a.s.}
\end{equation*}
This, combining with the fact that $(1+r)^\alpha-1\le \alpha 2^{\alpha-1}r$ for any $\alpha>1$ and $r\in[0,1]$,
guarantees that for all $t\ge T_2,$
\begin{equation*}
\frac{1}{t}\int_{\lfloor t\rfloor}^t|f(X_s)|\,\d s \lesssim \lfloor t\rfloor^{-\frac{3}{4}} \lesssim t^{-\frac{3}{4}},\quad \mbox{ a.s.}
\end{equation*}
As a result, the second statement in \eqref{EE2} is verifiable.

Based on the analysis above, we conclude that \eqref{EE1}  follows for the random time $T:=T_1+T_2.$
\end{proof}

\section{Proof of Theorem \ref{CLT}}\label{sec3}
Before we start to complete the proof of Theorem \ref{CLT}, we provide the following proposition, which
establishes the contractive property of transition kernels under the quasi-Wasserstein distance.

\begin{proposition}\label{pro2}
Under the assumptions of Theorem \ref{CLT},  for any $p\ge2 $, $\theta\in(0,1]$  and $\mu,\nu\in\mathscr P_{\psi_{p,\theta}}(\R^d)$, there exist constant $C^*\ge1,\lambda^*>0$ such that
\begin{equation}\label{EE6}
\mathbb W_{\psi_{p,\theta}}(\mu P_t,\nu P_t)\le C^*\e^{-\lambda^* t}\mathbb W_{\psi_{p,\theta}}(\mu  ,\nu  ),\qquad t\ge0,
\end{equation}
where
\begin{align*}
\mathscr P_{\psi_{p,\theta}}(\R^d):=\big\{\mu\in\mathscr P(\R^d):\mu(\psi_{p,\theta}(\cdot,{\bf0}))<\8\big\},
\end{align*}
and $\mathbb W_{\psi_{p,\theta}}$ denotes the quasi-Wasserstein distance $($see e.g. \cite[(4.3)]{HMS}$)$ induced by the cost function
$\psi_{p,\theta}$, introduced in \eqref{R1}. Moreover, $(X_t)_{t\ge0}$ solving \eqref{E1} has a unique invariant measure $\mu\in\mathscr P_{\psi_{p,\theta}}(\R^d)$.
\end{proposition}

\begin{proof}
Throughout the whole proof to be implemented, we still utilize  the coupling constructed in the proof of Proposition \ref{pro1}; see \eqref{E8} and \eqref{E10} for more details. In view of \eqref{E0}, \eqref{E4} and \eqref{E*6}, for any $p\ge2$ and $V_p(x):=1+|x|^p, x\in\R^d,$
 there are constants $C_1(p),C_2(p)>0$  such that
\begin{equation}\label{E*8}
(\mathscr L V_p)(x)\le -C_1(p)V_p(x)+C_2(p),\quad x\in\R^d,
\end{equation}
where $\mathscr L$ is the infinitesimal generator of \eqref{E4}.
For the parameters   $C_1(p),C_2(p)$ above,  the  set
\begin{equation*}
 \mathcal A_p:=\big\{(x,y)\in\R^d\times\R^d:C_1(p)\big(V_p(x)+V_p(y)\big)\le 4C_2(p)\big\}
\end{equation*}
is of finite length since the mapping $\R^d\ni x\mapsto V_p(x)$ is compact. Therefore, the quantity
\begin{equation*}
l_p^*:=1+\sup\big\{|x-y|:(x,y)\in  \mathcal A_p\big\}<\8
\end{equation*}
is well defined.

In order to achieve the exponential contractivity  \eqref{EE6}, it is vital to define two  auxiliary functions as below.
Define for $\theta\in(0,1],$
\begin{equation}\label{W10}
h(r)= \kk \int_0^{r}\e^{-\frac{\kk}{2}\int_0^u((\lambda +2K_2\kk^3) v+ 2K_1v^\alpha)\,\d v} \int_u^{l_p^*}v^\theta\e^{\frac{\kk}{2}\int_0^v((\lambda +2K_2\kk^3) l+ 2K_1l^\alpha)\,\d l}\,\d v \,\d u,\quad r\ge0,
\end{equation}
where
  $K_1>0$, $\kk>0$   and $\lambda>0$ were introduced in   \eqref{E0}, \eqref{E4} and \eqref{E*7}, respectively, and $\alpha\in(0,1)$ is the H\"older index associated with $b_0$. Moreover, we define   for $\theta\in(0,1],$
\begin{equation}\label{E*4}
f(r)=c^*(r\wedge l_p^*)^\theta+h(r\wedge l_p^*),\quad r\ge0,
\end{equation}
where
\begin{equation}\label{W11}
  c^*:=\frac{1}{\theta(\lambda+
4K_2\kk^3+2K_1r_0^{\alpha-1})}\I_{(0,1)}(\theta)\qquad \mbox{ with }\qquad  r_0:=1\wedge\Big(\frac{1-\theta}{ 2
K_1\kk}\Big)^{\frac{1}{1+\alpha}}.
\end{equation}

In the sequel, we shall fix the    function $f$ defined  in \eqref{E*4} and choose  the tuneable parameter
\begin{equation}\label{E*5}
\begin{split}
\vv:=1&\wedge \frac{1}{16C_2(p)(c^* +h'(0) (l_p^*)^{1-\theta})}\\
&\wedge\frac{1}{\big(2^{4+\frac{4}{p}}c^\star((p-2)^{1-\frac{2}{p}} \vee (p-1)^{1-\frac{1}{p}})(c^*\theta +h'(0) (l_p^*)^{1-\theta} )\big)^p},
\end{split}
\end{equation}
where
\begin{equation}\label{W1}
c^\star:=\big(\frac{1}{\kk} 2^{(p-2)\vee 1}(p-1)\big)\vee (2K_2 ^{\frac{1}{2}}\kk^2).
\end{equation}
Note that, for the case $p=2$,  the term $(p-2)^{1-\frac{2}{p}}$ should be  understood in the limit sense, that is, $\lim_{p\downarrow2}(p-2)^{1-\frac{2}{p}}=1$.
Moreover, in the following context,   for the sake of convenience,   we shall write
\begin{equation*}
\Psi(t):=f(|Z_t|)\big(1+\vv V_p(Y_t)+\vv V_p(\hat Y_t )\big),\quad t\ge 0.
\end{equation*}

By applying It\^o's formula, we deduce from   \eqref{E8} and \eqref{E10} that
\begin{equation}\label{E*1}
\begin{split}
\d \Psi(t)&=(1+\vv V_p(Y_t)+\vv V_p(\hat Y_t ))\I_{\{0<|Z_t|\le l_p^*\}}\d f(|Z_t|)\\
&\quad+\vv f(|Z_t|)\d (|Y_t|^p+|\hat Y_t|^p)+\vv\d\<|Y_\cdot|^p+|\hat Y_\cdot|^p,f(|Z_\cdot|)\>(t),\quad t<\tau,
\end{split}
\end{equation}
in which $Z_t:=Y_t-\hat Y_t$ and
$\<\xi,\eta\>(t)$ means the quadratic variation of stochastic processes $(\xi_t)_{t\ge0}$ and $(\eta_t)_{t\ge0}$. On the one hand,  applying the It\^o-Tanaka formula to \eqref{E15}, followed by taking \eqref{E0}, \eqref{E*7}, \eqref{E11} and $f''\le 0$ into consideration yields
\begin{equation*}
\begin{split}
\d f(|Z_t|)&\le \Big(\frac{1}{2 }f'(|Z_t|) \big(\lambda  |Z_t|+
4K_2\kk^3|Z_t|+ 2K_1|Z_t|^\alpha\big)+\frac{1}{\kk}f''(|Z_t|)\Big)\,\d t \\
&\quad+f'(|Z_t|)\Big(\big\<(\tilde{\si}(Y_t)-\tilde{\si}(\hat Y_t) )^*{\bf n}(Z_t ), \d \tilde{W}_t\big\>+\frac{1}{\ss {\kk/2}}\big\<{\bf n}(Z_t ),\d\hat{W}_t\big\>\Big),\quad t<\tau.
\end{split}
\end{equation*}
On the other hand, by applying It\^o's formula once more and
making use of  the Lyapunov condition \eqref{E*8}, we derive that
\begin{equation*}
\begin{split}
\d \big(|Y_t|^p+|\hat Y_t|^p\big)&\le \big\{-C_1(p) (V_p(Y_t)+V_p(\hat Y_t) )+2C_2(p)\big\}\,\d t\\
 &\quad+p\<|Y_t|^{p-2}\tilde{\si}(Y_t)^*Y_t+|\hat Y_t|^{p-2}\tilde{\si}(\hat Y_t)^*\hat Y_t,\d \tilde{W}_t\>\\
&\quad  +\frac{p}{\ss {2\kk}} \<|Y_t|^{p-2}Y_t+|\hat Y_t|^{p-2}\Pi_{Z_t} \hat Y_t, \d \hat{W}_t\>,\quad t<\tau.
\end{split}
\end{equation*}
Consequently,  combining the estimates on $\d f(|Z_t|)$ and $\d \big(|Y_t|^p+|\hat Y_t|^p\big)$ with \eqref{E*1} and
\begin{align*}
\frac{1}{|x-y|}\<x-y,|x|^{p-2}x+|y|^{p-2}\Pi_{x-y}y\>=\frac{1}{|x-y|}\<x-y,|x|^{p-2}x-|y|^{p-2}y\>,\quad x\neq y
\end{align*}
enables us to
derive that
\begin{equation*}
\d \Psi(t)
 \le\big(\Theta_1+\Theta_2\big)(Y_t,\hat Y_t)\,\d t +\d M_t,\quad t<\tau
\end{equation*}
for some underlying martingale $(M_t)$, where for any $x,y\in\R^d$ with $x\neq y,$
\begin{equation*}
\begin{split}
\Theta_1(x,y):&=(1+\vv V_p( x)+\vv V_p(y))\\
&\quad\times\Big(\frac{1}{2 }f'(|x-y|) \big(\lambda  |x-y|+
4K_2\kk^3|x-y|+2 K_1|x-y|^\alpha\big) +\frac{1}{\kk}f''(|x-y|)\Big)\\
 &\quad\times\I_{\{0<|x-y|\le l_p^*\}} +\vv  f(|x-y|)\big(-C_1(p)\big(V_p(x)+ V_p(y)\big)+2C_2(p)\big), \\
\Theta_2(x,y):&=p\vv f'(|x-y|)\Big(\frac{1}{\kk}\big||x|^{p-2}x-|y|^{p-2}y\big|\\
  &\qquad\qquad\qquad\quad+ \big|(\tilde{\si}(x)-\tilde{\si}(y) )^*\big(|x|^{p-2}\tilde{\si}(x)^*x+|y|^{p-2}\tilde{\si}(y)^*y\big)\big|
\Big).
\end{split}
\end{equation*}

In case  of
\begin{equation}\label{W3}
\Theta_1(x,y)+\Theta_2(x,y)
\le -\lambda^\star  f(|x-y|)\big(1+\vv V_p(x)+ \vv V_p(y)  \big),\qquad x,y\in\R^d,
\end{equation}
where
$$\lambda^\star:= \frac{1}{4(c^* +h'(0) (l_p^*)^{1-\theta} )}\wedge \frac{ C_1(p)\vv }{  1+2\vv  },$$
we then at once arrive at
\begin{equation*}
\d \Psi (t)
 \le-\lambda^*\Psi(t)\,\d t +\d M_t,\quad t<\tau.
\end{equation*}
Subsequently, via the It\^o formula, the estimate
\begin{equation*}
\E\big(\e^{\lambda^* (t\wedge \tau)}\Psi (t\wedge \tau)\big)\le  \Psi (0)
\end{equation*}
is available so the assertion \eqref{EE6} is attainable by taking advantage of the fact that
\begin{equation*}
\big(c^*\wedge f(l_p^*)\big)(1\wedge r^\theta)\le f(r)\le \big(f(l_p^*)\vee \big(c^*+h'(0)(l_p^*)^{1-\theta}\big)\big)(1\wedge r^\theta),
\end{equation*}
and $Z_t=0$ for $t\ge\tau. $
On the basis of the preceding analysis, it all boils down to  the confirmation of \eqref{W3} in order to achieve \eqref{EE6}.

In accordance with the definition of the function $h$ introduced in \eqref{W10},  a direct calculation reveals  that
 \begin{equation*}
\frac{1}{2 }h'(r)  (\lambda   r+
4K_2\kk^3 r+ 2K_1r^\alpha)+\frac{1}{\kk}h''(r)=-r^\theta,\qquad r\in(0,l_p^*]
 \end{equation*}
so that for all $r\in(0,l_p^*]$,
\begin{equation}\label{W12}
\begin{split}
 &\frac{1}{2 }f'(r)  (\lambda   r+
4K_2\kk^3 r+ 2K_1r^\alpha)+\frac{1}{\kk}f''(r)\\
 &= c^*\theta\big(\frac{1}{2 }  \big(\lambda  r^{\theta} +
4K_2\kk^3 r^\theta+ 2K_1r^{\theta+\alpha-1} \big)-\frac{1}{\kk} (1-\theta)r^{\theta-2}\big)-r^\theta.
 \end{split}
\end{equation}
By  noting that
 \begin{equation*}
 2K_1r^{\theta+\alpha-1} -\frac{1}{\kk} (1-\theta)r^{\theta-2}\le 0,\qquad r\in(0,r_0],
 \end{equation*}
 where $r_0>0$ was introduced in \eqref{W11},  we right away have
 \begin{equation*}
 \frac{1}{2 }f'(r)  (\lambda   r+
4K_2\kk^3 r+ 2K_1r^\alpha)+\frac{1}{\kk}f''(r)\le-\big(1-\frac{1}{2 } (\lambda +
4K_2\kk^3 )c^*\theta\big) r^{\theta},\quad r\in(0,r_0].
\end{equation*}
 Furthermore,  by invoking \eqref{W12} again,
 we apparently infer  from $\alpha\in(0,1)$   that for all  $r\in [r_0,l_p^*]$,
  \begin{equation*}
 \frac{1}{2 }f'(r)  (\lambda   r+
4K_2\kk^3 r+ 2K_1r^\alpha)+\frac{
 1}{\kk}f''(r)\le- \big(1-\frac{1}{2 }  c^*\theta\big(\lambda  +4 K_2\kk^3 +2K_1r_0^{\alpha-1} \big) \big)r^\theta.
  \end{equation*}
 Consequently, taking the choice of $c^*>0$ given in \eqref{W11} into consideration yields that
  \begin{equation*}
 \frac{1}{2 }f'(r)  (\lambda   r+
4K_2\kk^3 r+2 K_1r^\alpha)+\frac{1}{\kk}f''(r)\le- \frac{1}{2}r^\theta,\qquad r\in(0,l_p^*].
  \end{equation*}
 This definitely implies that
\begin{equation}\label{W13}
\begin{split}
\Theta_1(x,y)&\le-\frac{1}{2}|x-y|^\theta(1+\vv V_p( x)+\vv V_p(y))\I_{\{0<|x-y|\le l_p^*\}}\\
&\quad+\vv  f(|x-y|)\big(-C_1(p)\big(V_p(x)+ V_p(y)\big)+2C_2(p)\big),\qquad x,y\in\R^d.
\end{split}
\end{equation}

Next,
by means of \eqref{E*}, \eqref{E11} and \eqref{E4}, in addition to $h'(r)\le h'(0)$ for any $r\in[0,l_p^*],$ it follows from Young's inequality that for any $\alpha,\beta>0$ and $p>2,$
\begin{equation*}
\begin{split}
\Theta_2(x,y)&\le p\vv\phi( |x-y|) \Big(\frac{1}{\kk} 2^{(p-2)\vee1}(p-1) \big(|x|^{p-2}+|y|^{p-2}\big)+2{K_2}^{\frac{1}{2}}\kk^2 \big(|x|^{p-1} +|y|^{p-1} \big)\Big)\\
&\le  pc^\star \vv \phi( |x-y|)\big( |x|^{p-2}+|y|^{p-2}+ |x|^{p-1} +|y|^{p-1}  \big)\\
&\le   c^\star \phi( |x-y|)\Big( (p-2)\alpha\big(\vv |x|^p+\vv|y|^p+\frac{4}{p-2}\alpha^{-\frac{p}{2}}\vv\big)\\
&\qquad\qquad\qquad\quad+ (p-1)\beta\big(\vv|x|^p+\vv|y|^p+\frac{2}{p-1}\bb^{-p}\vv\big)\Big),\qquad x,y\in\R^d,
\end{split}
\end{equation*}
where $c^\star>0$ was defined as in \eqref{W1}, and $\phi(r):=c^*\theta r^{\theta}+h'(0)r,r\ge0.$
In particular, choosing $\alpha=\Big(\frac{4\vv}{p-2}\Big)^{\frac{2}{p}}$ and $\beta=\Big(\frac{ 4\vv}{p-1}\Big)^{\frac{1}{p}}$, respectively, and taking the alternative of $\vv$ given in \eqref{W1} and $l_p^*\ge1$ into account  leads to
\begin{equation}\label{W2}
\begin{split}
\Theta_2(x,y)
&\le   2^{1+\frac{4}{p}}c^\star\Big((p-2)^{1-\frac{2}{p}} \vee (p-1)^{1-\frac{1}{p}}\Big)\vv^{\frac{1}{p}} \phi( |x-y|)\big( 1+ \vv V_p( x)+\vv V_p(y) \big)\\
&\le 2^{1+\frac{4}{p}}c^\star\Big((p-2)^{1-\frac{2}{p}} \vee (p-1)^{1-\frac{1}{p}}\Big) \big(c^*\theta +h'(0) (l_p^*)^{1-\theta} \big)\vv^{\frac{1}{p}}|x-y|^{\theta}\\
&\quad\times \big( 1+ \vv V_p( x)+\vv V_p(y) \big)\\
&\le \frac{1}{8} |x-y|^\theta\big(  1+\vv V_p( x)+\vv V_p(y)  \big),\qquad x,y\in\R^d.
\end{split}
\end{equation}
Whereafter, for any $x,y\in\R^d$ with $|x-y|\le l_p^*$,   estimates \eqref{W13}  and  \eqref{W2}  yield that
\begin{equation}\label{W4}
\begin{split}
\Theta_1(x,y)+\Theta_2(x,y)&\le -\frac{3}{8}(1+\vv V_p( x)+\vv V_p(y))|x-y|^\theta+2C_2(p)\vv  \big(c^*|x-y|^\theta+h'(0)|x-y|\big)\\
&\le  -\frac{1}{4}(1+\vv V_p( x)+\vv V_p(y))|x-y|^\theta\\
&\le -\frac{1}{4(c^* +h'(0) (l_p^*)^{1-\theta} )} f(|x-y|)(1+\vv V_p( x)+\vv V_p(y)),
\end{split}
\end{equation}
where in the first inequality   we also used the fact that $h(r)\le h'(0)r$ for all $r\in[0,l_p^*],$  in the second inequality we employed
the choice of $\vv$ provided in \eqref{E*5}, and in the last inequality we exploited the fact that
\begin{equation*}
c^*r^\theta\le f(r)\le \big(c^* +h'(0) (l_p^*)^{1-\theta} \big)r^\theta,\qquad r\in[0,l_p^*].
\end{equation*}

For any $x,y\in\R^d$ with $|x-y|>l_p^* $ (which obviously indicates  $(x,y)\notin  \mathcal A_p$),  with the aid of  $f'(r)=0$
for any $r>l_p^*,$ we deduce from the notions of $\Theta_1$ and $\Theta_2$ that
\begin{equation}\label{W6}
\begin{split}
\Theta_1(x,y)+\Theta_2(x,y)&\le-\frac{1}{2}C_1(p)\vv  f(|x-y|)\big( V_p(x)+   V_p(y)   \big)\\
&=-\frac{C_1(p)\vv (V_p(x)+   V_p(y))}{2(1+\vv (V_p(x)+   V_p(y)))}  f(|x-y|)\big(1+\vv V_p(x)+ \vv V_p(y)  \big)\\
&\le -\frac{ C_1(p)\vv }{  1+2\vv  }  f(|x-y|)\big(1+\vv V_p(x)+ \vv V_p(y)  \big),
\end{split}
\end{equation}
where the last line  is due to $V_p\ge1$.

At length, \eqref{W3} is verifiable by combining \eqref{W4} with \eqref{W6} concerning  the cases $|x-y|\le l_p^* $ and $|x-y|> l_p^* $ for all $x,y\in\R^d$, respectively.

Once \eqref{EE6} is available, the existence and uniqueness of invariant probability measures in $\mathscr P_{\psi_{p,\theta}}(\R^d)$ can be derived by following exactly the line in \cite[Corollary 4.11]{HMS}.
\end{proof}

With Proposition \ref{pro2} at hand, we are in position to finish the
\begin{proof}[Proof of Theorem \ref{CLT}]
Inspired  essentially by the procedure in \cite{Sh}, we shall decompose the additive functional $\bar A_t^{f,x}: =\frac{1}{\ss t}\int_0^tf(X_s^x)\,\d s$ for $f\in  C_{p,\theta}(\R^d)$ with $\mu(f)=0$ into two parts, where the one part is concerned with the additive functional of a martingale under consideration, and the other part is the corresponding remainder term. In order to achieve the desired convergence rate in the CLT for the additive functional $\bar A_t^{f,x}$, we shall adopt the convergence rate concerning the CLT for martingales (see e.g.  \cite[Theorem 3.10]{HH}) to treat the martingale part involved, and meanwhile exploit the contractive property (i.e., \eqref{EE6}) to handle the remainder term.

 In following proof, we shall fix $f\in  C_{p,\theta}(\R^d)$ with $\mu(f)=0$. For $x\in\R^d$ and $t\ge1,$ let
\begin{equation*}
M_t^{f,x}=\int_0^t\{f(X_s^x)-(P_sf)(x)\}\,\d s+\int_t^\8\{(P_{s-t}f)(X_t^x)-(P_sf)(x)\}\,\d s.
\end{equation*}
It can readily be noted that the additive functional $\bar A_t^{f,x}$ can be rewritten as below:
\begin{equation}\label{W99}
\begin{split}
\bar A_t^{f,x} &=\frac{1}{\ss {\lfloor t\rfloor} }M_{\lfloor t\rfloor}^{f,x}+\bigg(\big({\lfloor t\rfloor}^{-\frac{1}{2}}-  t^{-\frac{1}{2}}\big)M_{\lfloor t\rfloor }^{f,x},\\
&\qquad\qquad\qquad\qquad+t^{-\frac{1}{2}} \Big( \int_{\lfloor t\rfloor }^tf(X_s^x)\,\d s+ \int_0^\8\big( (P_sf)(x)  - (P_sf)(X_{\lfloor t\rfloor }^x)\big)\,\d s\Big)\bigg)\\
&=:\bar M_t^{f,x}+R_t^{f,x}.
\end{split}
\end{equation}

Recall from \cite[Lemma 2.9]{Sh}  the basic fact that for any real-valued random variables $\xi,\eta$ and any $\aa>0,\sigma\ge0,$
\begin{equation*}
\sup_{z\in\R}\big|\P(\xi\le z)-\Phi_\sigma(z)\big|\le \sup_{z\in\R}\big|\P(\eta\le z)-\Phi_\sigma(z)\big|+\P(|\xi-\eta|>\aa)+c_\sigma\aa,
\end{equation*}
where $c_\sigma:=\frac{1}{\sigma\ss{2\pi}}\I_{\{\sigma>0\}}+2\I_{\{\si=0\}}$. Thus, the decomposition \eqref{W99} enables us to derive that
for any $\aa>0 $ and $\sigma\ge0$,
\begin{equation*}
\sup_{z\in\R^d}\big|\P\big(\bar A_t^{f,x}\le z\big)-\Phi_\sigma(z)\big|\le \sup_{z\in\R^d}\big|\P\big(\bar M_{\lfloor t\rfloor}^{f,x}\le z\big)-\Phi_\sigma(z)\big|+\P\big(|R_t^{f,x}|>\aa\big)+c_\sigma\aa .
\end{equation*}
Hence, the desired assertion \eqref{EE6} follows as soon as,   for any $\vv\in(0,1/4)$,  there exists a constant $C_\vv(x)>0$ such that
\begin{equation}\label{EE7}
\P\big(|R_t^{f,x} |>t^{-\frac{1}{4}} \big) \le C_\vv(x) t^{-\frac{1}{4}},\qquad \sup_{z\in\R^d}\big|\P\big(\bar M_{\lfloor t\rfloor}^{f,x}\le z\big)-\Phi_\sigma(z)\big|\le C_\vv(x) t^{-\frac{1}{4}+\vv}.
\end{equation}

Due to  $  {\lfloor t\rfloor}^{-\frac{1}{2}}-  t^{-\frac{1}{2}}<   t^{-\frac{1}{2}} $ for any $t\ge1$,
 it follows from Chebyshev's inequality that for $t\ge1, $
 \begin{equation}\label{W5}
 \begin{split}
\P\big(|R_t^{f,x}|>t^{-\frac{1}{4}}\big)
&\le  t^{\frac{1}{4}}\Big(\big({\lfloor t\rfloor}^{-\frac{1}{2}}-  t^{-\frac{1}{2}}\big)\E\Big|\int_0^{\lfloor t\rfloor }f(X_s^x)\,\d s\Big|+  t^{-\frac{1}{2}}\Theta_p(t,x)\Big),
\end{split}
\end{equation}
where
$$\Theta_p(t,x):=\int_{\lfloor t\rfloor }^t\E|f(X_s^x)|\,\d s+2 \int_0^\8\big(| (P_sf)(x)|  +\E\big|(P_sf)(X_{\lfloor t\rfloor }^x)\big|\big)\,\d s.$$
 For any $q\ge2$, applying It\^o's formula followed by taking advantage of  \eqref{E*8}  yields that for some constant $C_1(q)>0,$
\begin{equation}\label{W7}
\sup_{t\ge s }\E|X_t^x|^q\lesssim 1+\E|X_s^x|^q,\quad s\ge 0
\end{equation}
so that, for $f\in C_{p,\theta}(\R^d)$,
\begin{equation}\label{W7*}
\sup_{t\ge s}\E f(X_t^x)\lesssim 1+\E|X_s^x|^p \lesssim 1+|x|^p,\quad s\ge0.
\end{equation}
Accordingly, Proposition \ref{pro2}, together with  $\mu(f)=0$,   implies  that
\begin{equation}\label{W8}
\Theta_p(t,x)
\lesssim  1+|x|^p +\|f\|_{p,\theta} \int_0^\8\e^{-\lambda^* s}\big(1+|x|^p+\E|X_{\lfloor t\rfloor }^x|^p\big)\,\d s\lesssim 1+|x|^p.
\end{equation}
Further,  owing to the Markov property of $(X_t^x)_{t\ge0}$, Proposition \ref{pro2}  and \eqref{W7},  we deduce  that
\begin{equation*}
\begin{split}
\E\Big|\int_0^{\lfloor t\rfloor }f(X_s^x)\,\d s\Big|^2&=2\int_0^{\lfloor t\rfloor }\int_{s}^{ \lfloor t\rfloor  }\E\big( f(X_s^x)(P_{u-s}f)(X_s^x)\big)\,\d u\d s\\
&\lesssim \|f\|_{p,\theta}^2\int_0^{\lfloor t\rfloor }\int_{s}^{ \lfloor t\rfloor  }\e^{-\lambda^*(u-s)} \big( 1+\E|X_s^x|^{2p}\big)\,\d u\d s\\
&\lesssim \|f\|_{p,\theta}^2(1+|x|^{2p})\lfloor t\rfloor.
\end{split}
\end{equation*}
This further gives that
\begin{equation}\label{W9}
\begin{split}
\big({\lfloor t\rfloor}^{-\frac{1}{2}}-  t^{-\frac{1}{2}}\big)\E\Big|\int_0^{\lfloor t\rfloor }f(X_s^x)\,\d s\Big| &\lesssim \|f\|_{p,\theta} (1+|x|^{ p})\big(t^{\frac{1}{2}}-   \lfloor t\rfloor^{\frac{1}{2}}\big)t^{-\frac{1}{2}}\lesssim \|f\|_{p,\theta} (1+|x|^{ p}) t^{-\frac{1}{2}},
\end{split}
\end{equation}
where   the second inequality is valid thanks to   $t^{\frac{1}{2}}-   \lfloor t\rfloor^{\frac{1}{2}}\le (t-\lfloor t\rfloor)^{\frac{1}{2}}\le1.$ Subsequently,
plugging \eqref{W8} and \eqref{W9} back into \eqref{W5} guarantees the validity of the first statement in \eqref{EE7}.

We proceed  to  verify the second statement in \eqref{EE7}. In light of  Proposition \ref{pro2} and by invoking the semigroup property of $(P_t)_{t\ge0}$, it is easy to see that $(M_t^{f,x})_{t\ge0}$ is a square integrable martingale with the zero mean. Note  that $(M_n^{f,x})_{n\ge1}$ can be reformulated as follows: for any integer $n\ge1,$
\begin{equation*}
M_n^{f,x}=\sum_{i=1}^nZ_i^{f,x} \qquad \mbox{ with } \qquad Z_i^{f,x} :=M_i^{f,x}-M_{i-1}^{f,x}.
\end{equation*}
Trivially, according to the definition of $M_n^{f,x}$, we have
 for $1\le i\le n,$
\begin{equation*}
\begin{split}
Z_i^{f,x}
&=\int_{i-1}^if(X_s^x)\d s+R_f(X_i^x)  -R_f(X_{i-1}^x) \qquad \mbox{ with } \qquad R_f(x):=\int_0^\8 (P_sf)(x) \,\d s.
\end{split}
\end{equation*}

By means of the property of conditional expectation and the flow property of $(X_t^x)_{t \geq 0}$, it follows readily that
$$
\sum_{i=1}^n \mathbb{E}\big|Z_i^{f, x}\big|^2=\sum_{i=1}^n \mathbb{E}(\mathbb{E}(|Z_i^{f, x}|^2 | \mathscr{F}_{i-1}))=\sum_{i=1}^n \mathbb{E} \varphi_f(X_{i-1}^x),
$$
where
$$
\varphi_f(x):=\mathbb{E}\Big|\int_0^1 f(X_s^x)+R_f(X_1^x)-R_f(x)\Big|^2 .
$$
By applying Proposition \ref{pro2} and following exactly the routine of \cite[Lemma $4.1 \&$ Lemma 4.2]{BWYa}, we can deduce that $\varphi_f \in C_{2 p, \theta}(\mathbb{R}^d)$ satisfying
\begin{equation}\label{P2}
0 \le \mu(\varphi_f
)=2 \mu(f R_f)<\infty, \quad\|\varphi_f\|_{2 p, \theta} \lesssim\|f\|_{p, \theta}^2 .
\end{equation}
In addition, for any $q>\frac{1}{2}$ and $1 \le i \le n$, we apparently have,
$$
\mathbb{E}\big|Z_i^{f, x}\big|^{4 q} \le 3^{4 q-1}\big(\int_{i-1}^i \mathbb{E}|f(X_s^x)|^{4 q} \mathrm{~d} s+\mathbb{E}|R_f(X_i^x)|^{4 q}+\mathbb{E}|R_f(X_{i-1}^x)|^{4 q}\big),
$$
which, besides the fact that
$$
|R_f(x)| \lesssim\|f\|_{p, \theta}(1+|x|^p), \quad x \in \mathbb{R}^d,
$$
\eqref{W7}, and \eqref{W7*}, leads to
$$
\max _{1 \le i \le n} \mathbb{E}|Z_i^{f, x}|^{4 q} \lesssim 1+|x|^{4 p q}, \quad n \ge 1 .
$$
As a consequence, by applying the Berry-Esseen type estimate associated with martingales (see, for instance, \cite[Theorem 3.10]{BWY}) we derive that
\begin{equation}\label{WW1}
\begin{aligned}
& \sup _{z \in \mathbb{R}}\Big|\mathbb{P}\big(M_n^{f, x} / \sqrt{\mu(\varphi_f) n} \le z\big)-\Phi_1(z)\Big| \\
& \lesssim\bigg(n^{-q}+(\mu(\varphi_f))^{-2 q} \mathbb{E} {\Big|\frac{1}{n} \sum_{i=1}^n \varphi_f(X_{i-1}^x)-\mu(\varphi _f) \Big|}^{2 q}\bigg)^{\frac{1}{4 q+1}} \lesssim n^{-\frac{q}{4 q+1}},
\end{aligned}
\end{equation}
where in the last display we also utilized the fact that
\begin{equation}\label{ee}
\mathbb{E}\Big|\frac{1}{n} \sum_{i=1}^n \varphi_f(X_{i-1}^x)-\mu(\varphi_f)\Big|^{2 q} \lesssim n^{-q}
\end{equation}
by tracing exactly the line to derive \eqref{WW} and taking Proposition \ref{pro2} into account.
Concerning or the case $\sigma^2_*:={\mu(\varphi_f)}>0$, taking advantage of \eqref{WW1} gives that
$$
\begin{aligned}
\sup _{z \in \mathbb{R}^d}\Big|\mathbb{P}\big(\bar{M}_n^{f, x} \le z\big)-\Phi_{\sigma_*}(z)\Big| & =\sup _{z \in \mathbb{R}^d}\Big|\mathbb{P}\big(M_n^{f, x} /(\sqrt{n} {\sigma_*}) \le z / {\sigma_*}\big)-\Phi_{\sigma_*}(z)\Big| \\
& =\sup _{z \in \mathbb{R}^d}\Big|\mathbb{P}\big(M_n^{f, x} /(\sqrt{n} {\sigma_*}) \le z\big)-\Phi_1(z)\Big| \lesssim n^{-\frac{q}{4 q+1}}.
\end{aligned}
$$
As a result, the second statement in \eqref{EE7} follows directly for the case $\mu(\varphi_f)>0$.

 Note from Chebyshev's inequality that for a random variable $\xi$ and any $0 \neq z \in \mathbb{R}$,
$$
(1 \wedge|z|)\big|\mathbb{P}(\xi \le z)-{\I}_{[0, \infty)}(z)\big|=(1 \wedge|z|)\big(\mathbb{P}(\xi>z) \I_{\{z>0\}}+\mathbb{P}(-\xi \ge-z) \I_{\{z<0\}}\big) \le \mathbb{E}|\xi| .
$$
Therefore, with regard to the setting $\sigma^2_*={\mu(\varphi_f)}=0$, for any integer $n \ge 1$, we have
$$
\sup _{z \in \mathbb{R}}\big|\mathbb{P}\big(\bar{M}_n^{f, x} \le z\big)-\Phi_0(z)\big| \le \frac{1}{\sqrt{n}}\big(\mathbb{E}\big|M_n^{f, x}\big|^2\big)^{\frac{1}{2}} .
$$
This, together with
$$
\mathbb{E}|M_n^{f, x}|^2=2 \sum_{i=1}^n \sum_{j=i}^n \mathbb{E}(Z_i^{f, x} Z_j^{f, x})=\sum_{i=1}^n \mathbb{E}|Z_i^{f, x}|^2=\sum_{i=1}^n \mathbb{E} \varphi_f(X_{i-1}^x)
$$
by using the fact that $\mathbb{E}(Z_j^{f, x}|\mathscr{F}_i)=0$ for $j>i$, and \eqref{ee} with $\sigma^2_*={\mu(\varphi_f)}=0$, implies
$$
\sup _{z \in \mathbb{R}}\big|\mathbb{P}\big(\bar{M}_n^{f, x} \le z\big)-\Phi_0(z)\big| \lesssim n^{-\frac{1}{4}}.
$$
Whence, the second statement concerned with the case $\sigma^2_*={\mu(\varphi_f)}=0$ in \eqref{EE7} is verifiable.
\end{proof}

\section{Proof of Theorem \ref{thm2}}\label{sec4}
By following respectively the procedures to implement the proof of Theorems \ref{thm1} and \ref{CLT}, for the proof of Theorem \ref{thm2},
the key ingredient is to demonstrate
the contractive property under the quasi-Wasserstein distance. More precisely, we shall prove the following statement.

\begin{proposition}\label{pro3}
Under Assumptions $({\bf A}_b)$ and $({\bf H}_\si)$, for any $p\ge2 $, $\theta\in(0,1]$  and $\mu,\nu\in\mathscr P_{\psi_{p,\theta}}(\R)$, there exist constants $C^*\ge1,\lambda^*>0$ such that
\begin{equation}\label{E18}
\mathbb W_{\psi_{p,\theta}}(\mu P_t,\nu P_t)\le C^*\e^{-\lambda^* t}\mathbb W_{\psi_{p,\theta}}(\mu  ,\nu  ),\qquad t\ge0,
\end{equation}
where $\mu P_t$ denotes  the law of $X_t$ solving \eqref{W*1}  with the initial distribution $\mathscr L_{X_0}=\mu$. \eqref{E18} further implies that $(X_t)_{t\ge0}$ solving \eqref{W*1} has a unique invariant probability measure in $\mathscr P_{\psi_{p,\theta}}(\R)$.
\end{proposition}

Compared with the setups treated in Subsections \ref{subsection1} and \ref{subsection2},  the   outstanding feature of the framework in Subsection \ref{subsection3} is due to the discontinuity of drifts associated with SDEs under investigation. Thus, the approaches in tackling Propositions \ref{pro1} and \ref{pro2} cannot be applied directly. In view of this,
to handle the difficulty arising from the discontinuity of the drift term $b$, we  adopt the following transformation (see, for instance,  \cite{MY,MSY})
\begin{equation}\label{W*3}
U(x):=\sum_{i=1}^k \alpha_i(x-\xi_i)|x-\xi_i| \phi(({x-\xi_i})/{\dd}),\qquad x\in\R,
\end{equation}
where
\begin{equation*}
\alpha_i:=\frac{b(\xi_i-)-b(\xi_i+)}{2\si(\xi_i)^2},\qquad \phi(x):=(1-x^2)^4\I_{[-1,1]}(x),
\end{equation*}
and
\begin{equation}\label{W*2}
\dd:=
\begin{cases}
1\wedge\frac{\vv^\star}{32|\alpha_1|},\qquad \qquad \qquad \qquad \qquad \qquad\qquad \quad \quad\quad \quad\qquad\qquad  \,\,k=1,\\
1\wedge\frac{\vv^\star}{32\max\{|\alpha_1|,\cdots,|\alpha_k|\}}\wedge\Big( \frac{\vv^\star}{2}\min\{\xi_2-\xi_1,\cdots,\xi_k-\xi_{k-1}\}\Big),\quad k\ge2,
\end{cases}
\end{equation}
where the quantity $\vv^\star>0$ was introduced in \eqref{W*}.

The transformation $U$ given in \eqref{W*3} enjoys nice properties. In particular, $U$ and its derivative can be sufficiently small by
choosing appropriate parameter $\delta$ involved in the definition of $U.$ In terms of the definition of $\delta$ given in \eqref{W*2} and the prerequisite $\vv^\star\in(0,1/2]$, we have
\begin{equation}\label{EEE}
|U(x)|\le \dd^2 \sum_{i=1}^k|\alpha_i| {1}_{[\xi_i-\dd, \xi_i+\dd]}(x)\le \dd \max \{|\alpha_1|, \cdots,|\alpha_k|\}  \le \frac{\vv^\star}{1+\vv^\star}\le \frac{1}{3},\quad x\in\R.
\end{equation}
Moreover,
a direct calculation shows that the function  $U $   is differentiable such that
\begin{equation*}
U^{\prime}(x)=2 \sum_{i=1}^k \alpha_i|x-\xi_i|\big(1-((x-\xi_i) / \dd)^2\big)^3\big(1-5((x-\xi_i) / \dd)^2\big)^2 1_{[\xi_i-\dd, ~\xi_i+\dd]}(x),\quad x\in\R.
\end{equation*}
This obviously implies that
\begin{equation}\label{3E}
|U^{\prime}(x)| \le 32 \dd \sum_{i=1}^k|\alpha_i| {1}_{[\xi_i-\dd, \xi_i+\dd]}(x) \le 32 \dd \max \{|\alpha_1|, \cdots,|\alpha_k|\} \le  \vv^\star\le\frac{1}{2},
\end{equation}
by taking advantage of the alternative of $\delta$ and $\vv^\star\in(0,1/2].$ Furthermore, we define the transformation
\begin{align}\label{HH}
G(x):=x+U(x),\qquad x\in\R.
\end{align}
Apparently, \eqref{3E} yields that
\begin{equation}\label{W*9}
\frac{1}{2}\le 1-|U'(x)|\le  G'(x) \le 1+|U'(x)|\le \frac{3}{2}.
\end{equation}
%
Consequently, we conclude that the transformation $x\mapsto G(x)$
is a diffeomorphism.

Note that  $U'$ is differentiable on each interval $I_{i}, i\in\mathbb S_k $  so $U'$ is piecewise differentiable.
Then,  for $Y_t:=G(X_t)$,  applying It\^o's formula yields
\begin{equation}\label{W*7}
\d Y_t=\tilde{b}(Y_t)\,\d t+\tilde{\sigma}(Y_t)\,\d W_t,
\end{equation}
where
\begin{equation*}
\tilde{b}(x):=(G'b)(G^{-1}(x))+\frac{1}{2}(G''\sigma)(G^{-1}(x)),\quad \tilde{\sigma}(x): =( G'\si)(G^{-1}(x)),\quad x\in\R.
\end{equation*}
According to \cite[Lemma 2]{MSY}, the SDE \eqref{W*7} is  strongly well-posed via extending $U'': \cup_{i=1}^{k+1}I_i\to\R$ to $U:\R\to\R$ by in particular taking
 \begin{equation*}
U''(\xi_i)=2\Big(\aa_i+\frac{b(\xi_i+)-b(\xi_i-)}{\si(\xi_i)^2}\Big),\quad i\in\mathbb S_k.
\end{equation*}

With the preceding preliminaries, we start to complete the
\begin{proof}[Proof of Proposition \ref{pro3}]
 Below, let $(X_t^\mu)_{t\ge0}$ and $(Y_t^\mu)_{t\ge0}$ be the solutions to \eqref{W*1} and \eqref{W*7} with $\mathscr L_{X_0}=\mu\in\mathscr P(\R)$ and $\mathscr L_{Y_0}=\mu\in\mathscr P(\R),$ respectively.
Due to  the Kontorovich dual, besides $Y_t^{\mu\circ G^{-1}}=G(X_t^\mu)$ and $Y_t^{\nu\circ G^{-1}}=G(X_t^\nu)$ for $\mu,\nu\in\mathscr P_{\psi_{p,\theta}}(\R)$, we find that
\begin{equation}\label{E20}
\begin{split}
\mathbb W_{\psi_{p,\theta}}(\mu P_t,\nu P_t)
=\sup_{\|f\|_{\psi_{p,\theta}}\le 1}\big|\E (f\circ G^{-1})(Y_t^{\mu\circ G^{-1}})-\E (f\circ G^{-1})(Y_t^{\nu\circ G^{-1}})\big|.
\end{split}
\end{equation}
Next, by means of the mean value theorem and \eqref{W*9}, it follows that
\begin{equation}\label{E00}
\big|G^{-1}(x)-G^{-1}(y)\big|\le 2|x-y|,\qquad |G(x)-G(y)|\le \frac{3}{2}|x-y|,\quad x,y\in\R.
\end{equation}
This further implies that for any    $x,y\in\R$,
\begin{equation}\label{E16}
|(f\circ G^{-1})(x)-(f\circ G^{-1})(y)|+|(f\circ G )(x)-(f\circ G )(y)|\le C_p^* \|f\|_{\psi_{p,\theta}}\psi_{p,\theta}(x,y),
\end{equation}
where $$C_p^*:=\big(2^\theta\big(2^{2p-1}\vee\big(1+2^p|G^{-1}(0)|^p)\big) \big) \vee \big((3/2)^\theta\big((3^p/2)\vee   (1+2^p|G^{(0)}|^p)\big)\big).$$

 Subsequently,  via the Kontorovich dual   once more, along with \eqref{E20} and \eqref{E16},
 we deduce that
\begin{equation}\label{E21}
\mathbb W_{\psi_{p,\theta}}(\mu P_t,\nu P_t)
\le C_p^*\mathbb W_{\psi_{p,\theta}}\big((\mu\circ G^{-1}) \tilde{P}_t,(\nu\circ G^{-1}) \tilde{P}_t\big),\quad \mu,\nu\in\mathscr P_{\psi_{p,\theta}}(\R)
\end{equation}
where $(\mu\circ G^{-1}) \tilde{P}_t:=\mathscr L_{Y_t^{\mu\circ G^{-1}}}$. Provided that there exist constants $C^\star,\lambda^\star>0$ such that
\begin{equation}\label{E19}
\mathbb W_{\psi_{p,\theta}}\big( \mu  \tilde{P}_t, \nu \tilde{P}_t\big)\le C^\star\e^{-\lambda^\star t}\mathbb W_{\psi_{p,\theta}}\big( \mu  ,\nu \big), \quad \mu,\nu\in\mathscr P_{\psi_{p,\theta}}(\R)
\end{equation}
then we derive from \eqref{E21} that
\begin{equation*}
\begin{split}
\mathbb W_{\psi_{p,\theta}}(\mu P_t,\nu P_t) \le C_p^*C^\star\e^{-\lambda^\star t}\mathbb W_{\psi_{p,\theta}}\big( \mu\circ G^{-1},\nu\circ G^{-1}\big) \le \big(C_p^*\big)^2C^\star\e^{-\lambda^\star t}\mathbb W_{\psi_{p,\theta}}\big( \mu ,\nu \big),
\end{split}
\end{equation*}
where the second inequality is valid due to
 the Kontorovich dual   again and \eqref{E16}.

 Based on the analysis above, to achieve \eqref{E18}, it remains to claim that
  \eqref{E19} is verifiable.
By applying  Proposition \ref{pro2} with $b_0=0$, for the validity  of  \eqref{E19}, it amounts to proving that
\begin{enumerate}
\item[$({\bf H}_{\tilde{\sigma}})$] there exist    constants $K^*>0$  and $\kk^*\ge1$ such that for all $x,y\in\R,$
\begin{align*}
\frac{1}{\kk^*}\le\tilde{\sigma}(x)\le \kk^*,\qquad |\tilde{\sigma}(x)-\tilde{\sigma}(x)|\le K^*|x-y|;
\end{align*}

\item[$({\bf H}_{\tilde{b}})$]
there exist constants $\lambda_0,\lambda^*_0,C_{\lambda^*_0}>0$ such that for all $x\in\R,$
\begin{equation*}
2(x-y)(\tilde{b}(x)- \tilde{b}(y))\le \lambda_0(x-y)^2,\qquad x\,\tilde{b}(x)\le -\lambda^*_0x^2+C_{\lambda^*_0}.
\end{equation*}
\end{enumerate}

By recalling the definition of $\tilde{\sigma}$,
 we obtain from (${\bf H}_\si$) with $d=1$ and \eqref{W*9} that for any $x,y\in\R,$
\begin{equation*}
\begin{split}
\frac{1}{4\kk}\le \tilde{\sigma}(x)^2\le 4 \kk,\qquad
|\tilde{\sigma}(x)-\tilde{\sigma}(y)|
\le \ss\kk|(G'\circ G^{-1})(x)-(G'\circ G^{-1})(y)|+ 3 \ss{K_2}|x-y|.
\end{split}
 \end{equation*}
Notice that
\begin{equation*}
G''(x)=-2\alpha_i \psi_i(x),\quad  x\in(\xi_i-\vv,\xi_i);\quad G''(x)=2\alpha_i \psi_i(x),\quad x\in(\xi_i ,\xi_i+\vv),
\end{equation*}
and that, otherwise, $G''(x)=0$.
Thus, a straightforward calculation, besides the continuity of $G':\R\to\R$,  reveals  that there exists a constant $c_0>0$ such that
\begin{equation}\label{EE2}
|G'(x)-G'(y)|\le  c_0|x-y|,\quad x,y\in\R
\end{equation}
so by invoking \eqref{E00} there is a constant $c_1>0$ satisfying that
\begin{equation*}\label{E22}
|\tilde{\sigma}(x)-\tilde{\sigma}(y)|\le c_1|x-y|,\qquad x,y\in\R.
 \end{equation*}
 Therefore, we conclude that  the assertion $({\bf H}_{\tilde{\sigma}})$ follows.

By following the exact line to derive \cite[(A2')(i)]{MSY}, there exists a constant $c_2>0$ such that
\begin{equation*}
2(x-y)(\tilde{b}(x)-\tilde{b}(y))\le c_2|x-y|^2,\quad x,y\in\R.
\end{equation*}

Next,
by taking the definition of  $\tilde{b}$ into consideration, we find readily from \eqref{3E} and \eqref{EEE} that for some constant $c_3>0,$
\begin{align*}
x\tilde{b}(x)
&=\big( G^{-1}(x)+U(  G^{-1}(x))\big) \big( b( G^{-1}(x))+(U^{\prime}b)(G^{-1}(x))\big)+\frac{1}{2}x(G^{\prime \prime} \sigma)(G^{-1}(x))\\
&\le  G^{-1}(x)b( G^{-1}(x))+ |G^{-1}|(x)|U^{\prime}b|(G^{-1}(x))\\
&\quad+ |U|(  G^{-1}(x))   \big( |b|( G^{-1}(x))+|U^{\prime}b|(G^{-1}(x))\big)+\frac{1}{2}|x|\cdot|G^{\prime \prime} \sigma|(G^{-1}(x))\\
&\le G^{-1}(x)b( G^{-1}(x))+\vv^\star\big(1+|G^{-1}(x)|\big)| b|(G^{-1}(x))+c_3|x|,\qquad x\in\R,
\end{align*}
 where in the identity we used  $G(x)=x+U(x)$ and in the last inequality we employed \eqref{E4} and \eqref{EE2}. Whereafter,  \eqref{W*} yields that
\begin{align*}
x\tilde{b}(x)
&\le C_\star-\lambda^*(G^{-1}(x))^2+c_3|x|,\qquad x\in\R.
\end{align*}
This, together with the fact that
\begin{equation*}
|(G^{-1}(x)|^2= |x-U((G^{-1}(x))|^2\ge \frac{1}{2}|x^2|-|U((G^{-1}(x))|^2\ge\frac{1}{2}|x^2|-\frac{1}{9},\quad x\in\R,
\end{equation*}
by making use of  the basic inequality: $(a-b)^2\ge\frac{1}{2}a^2-b^2$ for $a,b\in\R$ and \eqref{EEE}, leads to
\begin{align*}
x\tilde{b}(x)
&\le c_4-c_5|x|^2,\quad x\in\R
\end{align*}
for some constants $c_4,c_5>0.$ As a consequence,   we reach the assertion $({\bf H}_{\tilde{b}})$.

Based on the contractivity \eqref{E19}, the transformed SDE \eqref{W*7} has a unique invariant probability measure $\nu\in\mathscr P_{\psi_{p,\theta}}(\R)$ by following the line of \cite[Corollary 4.11]{HMS}. Note that the transformation $G$ constructed above is a diffeomorphism.
Thus, via   integrals with respect to image measures, we conclude that 
 $\mu:=\nu\circ G\in\mathscr P_{\psi_{p,\theta}}(\R)$ is the unique invariant probability measure of $(X_t)_{t\ge0}$ solving \eqref{W*1}.
\end{proof}

Finally, with the aid of Proposition \ref{pro3}, we complete
 \begin{proof}[Proof of Theorem \ref{thm2}] Applying Proposition \ref{pro3} yields  that for any $f\in C_{p,\theta}(\R)$.
 \begin{equation}\label{ET}
\big|(P_tf)(x)-\mu(f)\big|\le \|f\|_{p,\theta}\mathbb W_{\psi_{p,\theta}}(\delta_x P_t,\mu P_t)\le C^*\e^{-\lambda^* t}(|x|^p+\mu(|\cdot|^p)),\quad t\ge0,~  x\in\R,
\end{equation}
With this estimate at hand, the strong LLN can be   verifiable  by tracing the line in the proof of Theorem \ref{thm1} and, in particular,  replacing $|X_t|$ in \eqref{WW}, \eqref{EEE1} and \eqref{EEE2} by $|X_t|^p$, respectively. Moreover, with the help of \eqref{ET}, the CLT can de derived by following exactly the procedure to tackle Theorem \ref{CLT} so we omit the corresponding details herein.
\end{proof}
\section*{Acknowledgements}
We would like to thank the AE and referees for careful comments and correc$\&$tions. The research of Jianhai Bao is supported by the National Key R\&D Program of China
(2022YFA1006004) and NSF of China (No. 12071340).

\end{document}